\documentclass[11pt]{article}

\usepackage[utf8]{inputenc}
\usepackage{geometry}
\usepackage{graphicx}

\usepackage{booktabs,bm}
\usepackage{array}
\usepackage{multirow}
\usepackage{paralist}
\usepackage{verbatim} 
\usepackage{subfig}
\usepackage[usenames,svgnames,dvipsnames]{xcolor}
\usepackage{amsmath}
\usepackage{amssymb}
\usepackage{amsthm}
\usepackage{amsbsy}
\usepackage{mathrsfs}
\usepackage{tikz}
\usepackage{tikz-cd}
\usetikzlibrary{calc,fadings,decorations.pathreplacing}
\usepackage{pgfplots}
\pgfplotsset{small,compat=newest}
\usepgfplotslibrary{dateplot}
\usepackage{pgf}
\usetikzlibrary{arrows}
\usetikzlibrary{matrix}
\usetikzlibrary{pgfplots.groupplots}
\usepackage[round]{natbib}
\usepackage{float}
\usepackage{subfig}
\usepackage{babel}
\usepackage[pdftex,plainpages=false,colorlinks=true,citecolor=RoyalBlue,linkcolor=Red,urlcolor=black,filecolor=black,bookmarksopen=true]{hyperref}
\usepackage{setspace}
\usepackage{enumitem}
\usepackage[ruled,linesnumbered]{algorithm2e}
\RestyleAlgo{ruled}
\SetKwComment{Comment}{/* }{ */}
\hypersetup{pdftitle={Edge multiscale finite element methods for semilinear parabolic problems with heterogeneous coefficients},
 pdfauthor={Leonardo A. Poveda, Shubin Fu, 
 Guanglian Li and Eric T. Chung}}
 
\numberwithin{equation}{section}

\usepackage{fancyhdr}
\usepackage[titles,subfigure]{tocloft} 

\newcommand{\norm}[1]{\left\|#1\right\|}

\newcommand{\normL}[2]{\norm{#1}_{#2}}

\newtheorem{thm}{Theorem}[section]
\newtheorem{prp}[thm]{Proposition}
\newtheorem{lem}[thm]{Lemma}

\newtheorem{asp}[thm]{Assumption}

\theoremstyle{definition}

\newtheorem{example}[thm]{Example}

\theoremstyle{remark}
\newtheorem{rem}[thm]{Remark}
\numberwithin{equation}{section}

\def\ie{\emph{i.e.}}
\def\eg{\emph{e.g.}}
\def\spn{\mathrm{span}}

\def\ms{\mathrm{ms}}

\def\dive{\mathrm{\nabla\cdot}}
\def\e{\small{\mathrm{e}}}
\def\feac{\quad\mbox{for each }}
\def\fora{\quad\mbox{for all }}

\def\Or{{\cal O}}

\def\p{\partial}

\def\pt{\p_{t}}
\def\ptt{\p^{2}_{tt}}
\def\gd{\nabla}

\def\k{\kappa}

\def\b{\boldsymbol{\beta}}
\def\bo{\b_{1}}
\def\bt{\b_{2}}

\def\uo{u_{1}}
\def\ut{u_{2}}

\def\wk{\widetilde{\k}}

\def\x{\mathrm{x}}
\def\R{\mathbb{R}}

\def\N{\mathbb{N}}

\def\Cx{\mathbb{C}}
\def\V{\mathrm{V}}

\def\H{\mathrm{H}}
\def\L{\mathrm{L}}
\def\W{\mathrm{W}}
\def\b{\boldsymbol{\beta}}
\def\VeI{\V^{I}_{\ell}}
\def\WeI{\W^{I}_{\ell}}
\def\Ho{\H^{1}}
\def\Hoz{\H^{1}_{0}}
\def\HoD{\Ho(D)}
\def\HozD{\Hoz(D)}

\def\HokD{\H^{1}_{\k}(D)}
\def\Hok{\L^{2}_{\k}}

\def\Lt{\L^{2}}
\def\LtD{\Lt(D)}
\def\Ltk{\Lt_{\k}}
\def\LtkD{\Ltk(D)}
\def\Linf{\L^{\infty}}
\def\LinfD{\Linf(D)}
\def\Vh{\V_{h}}

\def\Vmslew{\V_{\ms,\ell}}

\def\Kj{K_{j}}
\def\oi{\omega_{i}}
\def\TH{{\cal T}_{H}}
\def\Th{{\cal T}_{h}}
\def\FH{{\cal F}_{H}}
\def\Fh{{\cal F}_{h}}
\def\Po{{\cal P}_{1}}

\def\uh{u_{h}}
\def\vh{v_{h}}

\def\ums{u_{\ms}}

\def\vmsl{v_{\ms,\ell}}
\def\umsl{u_{\ms,\ell}}
\def\unmsl{\umsl^{n}}
\def\unpmsl{\umsl^{n+1}}
\def\ukmsl{\umsl^{k}}

\def\wmsl{w_{\ms,\ell}}
\def\umslew{u_{\ms,\ell}}

\def\uiI{u^{i,\mathrm{I}}}
\def\uiII{u^{i,\mathrm{II}}}
\def\uI{u^{\mathrm{I}}}
\def\uII{u^{\mathrm{II}}}
\def\uIImsl{\uII_{\ms,\ell}}

\def\Lc{{\cal L}}
\def\Linv{\Lc^{-1}}
\def\Li{\Lc_{i}}
\def\Liinv{\Li^{-1}}

\def\Ihpi{{\cal I}^{\p,i}_{h}}
\def\Or{{\cal O}}
\def\Pl{{\cal P}_{\ell}}
\def\Pil{{\cal P}_{i,\ell}}
\def\dt{\Delta t}
\def\Im{I}
\def\Lh{L_{h}}

\def\Lms{L_{\ms}}
\def\Rh{R_{h}}
\def\Rms{R_{\ms}}
\def\Pms{P_{\ms}}

\def\unmsl{\umsl^{n}}

\def\unpmsl{\umsl^{n+1}}

\def\ukms{\ums^{k}}
\def\ukmsl{\umsl^{k}}

\newcommand{\noteLeo}[1]{{\color{DarkBlue}{#1}}}

\newcommand{\dx}{{\mathrm{d}\x}}

\newcommand{\roma}{\mathrm{I}}
\newcommand{\romb}{\mathrm{II}}

\newcommand{\Cov}{C_{\mathrm{ov}}}

\title{Edge multiscale finite element methods for semilinear parabolic problems with heterogeneous coefficients}

\author{Leonardo A. Poveda\thanks{College of Science, Mathematics and Technology, Wenzhou-Kean University, Wenzhou, Zhejiang 325060, P. R. China (lpovedac@wku.edu.cn)} \and Shubin Fu\thanks{Eastern Institute for Advanced Study, Eastern Institute of Technology, Ningbo, 
Zhejiang 315200, P. R. China (sfu@eitech.edu.cn)} \and Guanglian Li\thanks{Department of Mathematics, University of Hong Kong, Pokfulam, Hong Kong (lotusli@maths.hku.hk).} \and Eric T. Chung\thanks{Department of Mathematics, The Chinese University of Hong Kong, Hong Kong (eric.t.chung@cuhk.edu.hk)}}

\begin{document}
\maketitle

\begin{abstract}
\pretolerance = 2000
We develop a new spatial semidiscrete multiscale method based upon the edge multiscale methods to solve semilinear parabolic problems with heterogeneous coefficients and smooth initial data. This method allows for a cheap spatial discretization, which fails to resolve the spatial heterogeneity but maintains satisfactory accuracy independent of the heterogeneity. This is achieved by simultaneously constructing a steady-state multiscale ansatz space with certain approximation properties for the evolving solution and the initial data. The approximation properties of the multiscale ansatz space are derived using local-global splitting. A fully discrete scheme is analyzed using a first-order explicit exponential Euler scheme. We derive the error estimates in the $\Lt$-norm and energy norm under the regularity assumptions for the semilinear term. The convergence rates depend on the coarse grid size and the level parameter. Finally, extensive numerical experiments are carried out to validate the efficiency of the proposed method.
\end{abstract}

\section{Introduction}
\pretolerance = 2000

We consider in this paper the efficient multiscale method for a semilinear parabolic equation with heterogeneous coefficient $\k$ in the form
\begin{align*}
\pt u - \dive(\k\gd u) +\b\cdot\gd u = R(\x,u) 
\end{align*}
with certain boundary conditions and initial conditions. Here, $R(\x,u)$ denotes the reaction term that can be nonlinear. This model type occurs in many applications, such as the phase field problems in the Allen-Cahn equation, which describes phase transition and separation \citep{Allen1979microscopic}. Other examples include the Ginszburg-Landau equation that studies the behavior of superconductivity \citep{caliari2024efficient}, the Navier-Stokes equations in highly heterogeneous media, which are used for the study of fluid dynamics and multiphase problems \citep{muljadi2015nonconforming}.

Classical numerical methods, such as finite element methods, resolve the coefficient's heterogeneity to obtain a solution with sufficient accuracy, resulting in huge computational complexity. 
During the last few decades, many multiscale model reduction techniques have been developed and analyzed to make computational complexity independent of the multiple scales, such as multiscale finite element methods \citep{hou1997multiscale}, multiscale variational methods \citep{ Hughes1998variational}, heterogeneous multiscale methods \citep{ weinan2003heterogeneous}, multiscale mortar methods \citep{Arbogast2007multiscale}, localized orthogonal decomposition methods \citep{malqvist2014localization}, generalized multiscale finite element methods \citep{efendiev2013generalized,chung2023multiscale}. These methods have demonstrated extremely satisfactory numerical results for various problems and have become increasingly popular. In this work, we will consider Edge Multiscale Methods (EMsFEM), which was initially proposed in \cite{li2019convergence,fu2019edge} and has been applied to several linear problems with heterogeneous coefficients \cite{fu2021wavelet,li2021wavelet,fu2023wavelet,li2024wavelet}. Its main idea is to first construct a local splitting of the solution in a sequence of overlapping subdomains, and then establish the global splitting by means of the partition of unity functions \cite{melenk1996partition}. In this manner, the global error estimate is determined by the local error estimate in each subdomain. Consequently, determining the local error estimate is a crucial task in the analysis. Its proof is inspired by the transposition method that provides {\em a priori }estimate in weighted $L^2$-norm of the homogeneous elliptic problem with nonhomogeneous Dirichlet boundary data in weighted $\Lt$-norm \cite{MR0350177}. The semilinearity makes the local splitting nontrivial, and the multiscale basis functions potentially rely on the nonlinear term. However, upon assuming the boundedness of the first-order partial derivative of the nonlinear term with respect to the solution, we prove that the multiscale space is independent of the semilinear term, and thus, one multiscale space is sufficient for all time steps. The local error estimate is presented in Lemma \ref{lem:Pl-interp}. Based upon this, we can derive an error estimate for the elliptic projection in Lemma \ref{lem:riesz}, then a standard error estimate yields the semidiscrete error estimate in Theorem \ref{thm:theta-estimate}.


Moreover, standard numerical approaches such as implicit methods have been proposed and analyzed for temporal discretization, which are attractive due to their unconditional stability. However, they involve solving nonlinear equations at each time step and thus are computationally expensive. In contrast, Exponential integrators are robust, with the linear part handled exactly and the non-linear part explicitly \citep{hochbruck2010exponential,caliari2024accelerating,contreras2023exponential,poveda2024second}, and we explore their performance for this semilinear parabolic problem. We present a rigorous convergence analysis of the proposed method under smooth initial data and certain regularity of the semilinear term in Theorem \ref{thm:main-result}, validated by extensive numerical experiments. 

The outline of the paper is as follows. Section~\ref{sec:problem-statement} is dedicated to the problem of the model and its spatial discretization. The multiscale edge space is constructed in Section~\ref{sec:edge-multiscale}.  Section~\ref{sec:hierarchical} presents the construction of the hierarchical bases. In Section \ref{sec:full-discretization}, we present the explicit exponential method. Under appropriate assumptions of the exact solution and the nonlinear reaction term, we show the fully discrete error analysis of the proposed approach. Numerical experiments are provided in Section~\ref{sec:numerical-exp}. Finally, conclusions and final comments are drawn in Section~\ref{sec:conclusion}.

\section{Problem setting}
\label{sec:problem-statement}
We are concerned with the semilinear parabolic problem
\begin{equation}\label{eq:main-prob}
\begin{aligned}
\pt u - \dive(\k\gd u) +\b\cdot\gd u &= R(\cdot,u) &&\text{ in } D\times (0,T]\\
u(\cdot,0)&=u_{0}&&\text{ in } D\\
u(\cdot,t)&=0&&\text{ on }\partial D\times (0,T].
\end{aligned}
\end{equation}
Let $D$ be a domain with $C^{1,\alpha}$, $(0<\alpha<1)$ boundary $\p D$, and $\{D_{i}\}_{i=1}^{m}\subset D$ be $M$ pairwise disjoint strictly convex open subsets, each with a $C^{1,\alpha}$ boundary $\Gamma_{i}:=\p D_{i}$, and denote $D_u{0}=D\setminus\overline{\bigcup_{i=1}^{m}D_{i}}$. Let the permeability coefficient $\k$ be a piecewise regular function defined by
\[
\k(\x)=\begin{cases}
\eta_{i}(\x),&  \mbox{in }D_{i},\\
1, & \mbox{in }D_{0}.
\end{cases}
\]
Here, $\eta_{i}\in C^{0,\nu}(\overline{D}_{i})$ with $\nu\in(0,1)$ for $i=1,\dots,m$. We denote $\eta_{\min}:=\min_{i}\{\|\eta_{i}\|_{C^0(D_i)}\}\geq 1$ and $\eta_{\max}:=\max_{i}\{\|\eta_{i}\|_{C^{0}(D_{i})}\}$. We further assume that $\eta_{\min}\gg 1$, and $\eta_{\min}$ has the same magnitude as $\eta_{\max}$. $\b\in[\LinfD]^{d}$ represents the advective vector field, which is incompressible, \ie, $\gd \cdot\b=0$. $R(\x,u)$ is often interpreted as a reaction term (possibly non-linear), assuming smooth on $D\times\R$, and $R(\x,\cdot)$ is uniformly Lipschitz continuous on $D$. 

We assume the initial data $u_{0}\in \dot{\H}^{2}(D)$, with 
\begin{align*}
\dot{\H}^{2}(D):=\left\{v\in \V:=\HozD: \mathcal{L}v:=-\nabla\cdot(\kappa\nabla v)+\b\cdot\gd v\in\LtD\right\}.
\end{align*}


Throughout this paper, $x\preceq y$ means a positive constant $C$ exists independent of $\kappa$ and mesh size $h$ to be introduced, such that $x\leq Cy$, and $x\simeq y$ means $y\preceq x \preceq y$.

Next, we recall the following assumptions related to the mild growth condition for the reaction term $R$ and the regularity estimates for the solution $u$, see \cite{thomee2006galerkin}.
\begin{asp}
\label{asp:01}
The function $R(\x,v)$ grows mildly with respect to $v$, i.e., there exists a number $p>0$ for $d=2$ or $p\in(0,2]$ for $d=3$ such that,
\begin{equation}
\label{eq:asp1}
\left|\frac{\partial R}{\partial v}(\x,v)\right|\preceq 1 + |v|^{p},\quad\feac v\in \mathbb{R}\text{ and }\x\in D.
\end{equation}
Moreover, $R(\x,0)=0$ for all  $\x\in D$.
\end{asp}
\begin{lem}[{\citealp[Chapter 14]{thomee2006galerkin}}]
\label{lem:02}
Let $u_0\in \dot{\H}^{2}(D)$, and let Assumption \ref{asp:01} hold. Then Problem \eqref{eq:main-prob} is well-posed, and its solution $u(\cdot,t)$ has the following regularity properties,
\begin{subequations}
\begin{align}
\|\mathcal{L}u(\cdot,t)\|_{\LtD}&\preceq 1,&&\text{ for }t\in[0,T],\\
\|\pt u(\cdot,t)\|_{\LtD}&\preceq 1,&&\text{ for }t\in[0,T],\\
\|\mathcal{L}^{1/2}\pt u(\cdot,t)\|_{\LtD}&\preceq t^{-1/2},&&\text{ for }t\in(0,T]\label{eq:3-lem2}.
\end{align}
\end{subequations}
\end{lem}
\begin{lem}[{\citealp[Lemma 2.1]{Larsson92}}]
\label{lem:local-R}
Suppose the function $R$ satisfies Assumption \ref{asp:01},  then $R$ is locally Lipschitz continuous in a strip along the exact solution $u(\cdot,t)$, meaning that for any fixed constant $C>0$, there holds
\begin{subequations}
\begin{align}
\|R(\cdot,v)-R(\cdot,w)\|_{\H^{-1}(D)}&\preceq \|v-w\|_{\LtD}\label{eq1:lem2}\\
\|R(\cdot,v)-R(\cdot,w)\|_{\LtD}&\preceq \|v-w\|_{\HoD}\label{eq2:lem2}
\end{align}
\end{subequations}
for any $t\in[0,T]$ and $v,w\in\V$ satisfying $\max\left\{\|(v-u(\cdot,t)\|_{\HoD},\|w-u(\cdot,t)\|_{\HoD}\right\}\leq C$, where the hidden constants may depend on $C$.
\end{lem}

Next, we introduce the weak formulation of \eqref{eq:main-prob}, which reads as seeking $u(\cdot,t)\in\V$ such that
\begin{equation}
\label{eq:weak-prob}
\begin{aligned}
(\pt u,v)+A(u,v)&=(R(\cdot,u),v),\quad\text{ for all } v\in \V,\\
u(\cdot,0)&= u_{0},
\end{aligned}
\end{equation} 
with $(\cdot,\cdot)$ denoting the inner product in $L^2(D)$,  $A(u,v):=a(u,v)+(\b\cdot\gd u,v)$, and 
\[
a(u,v):=\int_{D}\k\gd u\cdot\gd v\dx.
\]
Since that $\gd\cdot\b=0$, the bilinear form $A(\cdot,\cdot)$ in \eqref{eq:weak-prob} is $\V$-elliptic, \ie,
\begin{align}\label{eq:coercive}
A(v,v)=\|\nabla v\|^{2}_{\LtkD},\feac v\in\V.
\end{align}
Here, $\|\nabla v\|^{2}_{\LtkD}:=a(v,v)$.

An application of the Friedrichs' inequality implies the boundedness of the bilinear form $A(\cdot,\cdot)$,
\[
|A(u,v)|\leq C\|\nabla u\|_{\LtkD}\|\nabla v\|_{\LtkD},
\] 
where the positive constant $C$ depends on diameter of the domain $D$ and $\|\b\|_{L^{\infty}(D)}$. 

\section{The edge multiscale method}
\label{sec:edge-multiscale}
We present in this section the edge multiscale method to solve \eqref{eq:main-prob} for $d=2$ only for simplicity. The generalization of this method to $d=3$ can be followed similarly. The main idea of this approach is that first, we generate a coarse mesh $\mathcal{T}_H$ on the spatial domain $D$, which cannot resolve the microscale feature in the coefficient $\kappa$ and then introduce a sequence of overlapping subdomains using $\mathcal{T}_H$. Note that standard numerical methods cannot produce a reasonable solution under this coarse mesh $\mathcal{T}_H$ due to the pre-asymptotic effect. Second, we derive a local splitting of the solution on each subdomain by a summation of a local bubble function and a local Harmonic extension function. In contrast, the former can be solved locally, and the latter can be defined over the internal edges of the coarse mesh $\mathcal{F}_H$. This local splitting induces naturally a global splitting using the partition of unity functions. Third, we propose to use hierarchical bases up to level $\ell$ as the ansatz space for the solution on $\mathcal{F}_H$. The most crucial component in the theoretical analysis is to transfer the approximation properties over the coarse skeleton $\mathcal{F}_H$ to the internal of each subdomain without extra regularity assumptions on the solution. This was first proved in 
 \cite[Appendix A]{li2019convergence} by defining a very weak solution in each subdomain and establishing an {\em a priori} estimate inspired by the transposition method in \cite{MR0350177}. 
\subsection{Discretization}
First, we introduce a cheap discretization of the computational domain $D$ that fails to resolve the multiple scales in the heterogeneous coefficient $\k$. Let $\TH$ be a regular partition of the domain $D$ into quasi-uniform quadrilaterals with mesh size $H$. Let $\FH:=\cup_{T\in\TH}\partial T\backslash\partial D$ be the collection of internal edges in $\TH$. The set of nodes of $\TH$ are denoted by $\{\x_{i}\}_{i=1}^{N}$, with $N$ being the total number of coarse nodes. The coarse neighborhood associated with the node $\x_{i}$ is denoted by
\[
\oi:=\overline{\bigcup\left\{\Kj\in\TH:\x_{i}\in\overline{K}_{j}\right\}}.
\]

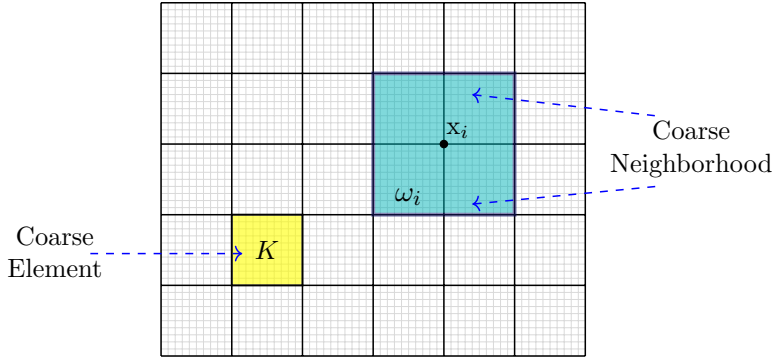
\begin{figure}[h!]
\centering
\begin{tikzpicture}[scale=4.7]
\draw[step=0.02cm,color=Gray, line width = 0.01mm,opacity=0.3] (0,0) grid (1.2,1.0);
\draw[step=0.2cm,color=Black,line width = 0.2mm] (0,0) grid (1.2,1.0);
\draw [draw=MidnightBlue, line width=1.5pt, fill=BlueGreen, opacity=0.50] (0.6,0.4) rectangle (1,0.8);
\draw [draw=black, line width=0.3mm, fill=Yellow, opacity=0.6] (0.2,0.2) rectangle (0.4,0.4);
\node at (0.3,0.3) {\small{$K$}};
\filldraw (0.8,0.6) circle (0.3pt);
\node at (0.84,0.64) {\small{$\x_{i}$}};
\node at (0.7,0.45) {$\oi$};
\node at (1.5,0.64) {\centering\small{\mbox{Coarse}}};
\node at (1.5,0.54) {\centering\small{\mbox{Neighborhood}}};
\node at (-0.3,0.34) {\centering\small{\mbox{Coarse}}};
\node at (-0.3,0.25) {\centering\small{\mbox{Element}}};
\coordinate (A1) at (0.88,0.74);
\coordinate (A2) at (1.4,0.68);
\coordinate (A3) at (0.88,0.43);
\coordinate (A4) at (1.4,0.48);
\coordinate (K1) at (0.23,0.29);
\coordinate (K2) at (-0.2,0.29);
\draw [<-,line width=0.2mm,dashed,Blue] (A1) -- (A2);
\draw [<-,line width=0.2mm,dashed,Blue] (A3) -- (A4);
\draw [<-,line width=0.2mm,dashed,Blue] (K1) -- (K2);
\end{tikzpicture}
\caption{2-d coarse grid $\mathcal{T}_H$, where $K$ and $\oi$ denote a coarse element and the coarse neighborhood associated with the node $\x_{i}$.}
\label{fig:grid}
\end{figure}
We defined the overlapping constant $\Cov$ by
\begin{equation}
\label{eq:overlapcond}
C_{\mathrm{ov}}:=\max_{K\in\TH}\#\{\x_{i}:K\in\oi,\mbox{for }i=1,\dots,N\}.
\end{equation} 
Next, we introduce the partition of unity $\{\chi_{i}\}_{i=1}^{N}$ subordinate to the sequence of overlapping subdomains $\{\oi\}_{i=1}^{N}$ that satisfies 
\begin{equation*}
\begin{aligned}
{\text{supp}(\chi_{i})}\subset\bar{\omega}_{i},\;
\sum\limits_{i=1}^{N}\chi_{i}&=1 \text{ in } D,\;
 \| \chi_{i}\|_{\Linf(\oi)}\leq C_{\infty},\quad
\|\nabla \chi_{i}\|_{\Linf(\oi)}\leq C_{\text{G}}H^{-1}
\end{aligned}
\end{equation*}
with $C_{\infty}$ and $C_{\text{G}}$ being positive constants independent of $\kappa$, $h$ and $H$. 

Then, we define the weighted coefficient:
\begin{equation}
\label{eq:wk}
\wk=H^{2}\k\sum_{i=1}^{N}|\gd\chi_{i}|^{2},
\end{equation}
and the weighted $\LtD$ space
\begin{align*}
\Lt_{\wk}(D)&:=\left\{w:\|w\|^{2}_{\Lt_{\wk}(D)}:=\int_{D}\wk w^{2}d\x<\infty\right\}.
\end{align*}

\subsection{Local-global splitting}
Note that the solution $u$ of problem \eqref{eq:main-prob} satisfies 
\[
\Li u:=\dive(\k\gd u)+\b\cdot\gd u = R(\cdot,u)-\pt u,\quad\mbox{in }\oi,
\]
which can be split into the summation of two parts,
\begin{equation}
\label{eq:split-u}
u|_{\oi} = \uiI + \uiII.
\end{equation}
Here, each component $\uiI$ and $\uiII$ are given by
\begin{equation}
\label{eq:uiI-prob}
\left\{
\begin{aligned}
\Li u^{i,\roma} &= R(\cdot,u)-\pt u&&\mbox{in }\oi,\\
u^{i,\roma} &= 0&&\mbox{on }\p\oi,
\end{aligned}
\right.
\end{equation}
and
\begin{equation}
\label{eq:uiII-prob}
\left\{
\begin{aligned}
\Li u^{i,\romb} &= 0&&\mbox{in }\oi,\\
u^{i,\romb} &= u&&\mbox{on }\p\oi\backslash\partial D\\
u^{i,\romb} &= 0&&\mbox{on }\p\oi\cap\partial D.
\end{aligned}
\right.
\end{equation}
Observe that $\uiI$ contains the local information that can be solved locally, and $\uiII$ encodes the global information over the inner coarse skeleton $\FH$. This local splitting induces a global splitting of $u$ by means of the partition of unity functions, 
\begin{equation}
\label{eq:decomp-u}
u = \left(\sum_{i=1}^{N}\chi_{i}\right)u=\sum_{i=1}^{N}\chi_{i}u|_{\oi}
=\sum_{i=1}^{N}\chi_{i}(\uiI+\uiII):=\uI + \uII,
\end{equation}
where
\begin{equation}
\label{eq:def-uII}
\uI =\sum_{i=1}^{N}\chi_{i}\uiI\text{ and } \uII = \sum_{i=1}^{N}\chi_{i}\uiII.
\end{equation}
\subsection{Hierarchical bases}
\label{sec:hierarchical}
Next, we introduce the hierarchical bases on the unit interval $I:=[0,1]$, which facilitates hierarchically splitting the space $\Lt(I)$.

Let the level parameter and the mesh size be $\ell$ and  $h_{\ell}=:2^{-\ell}$ with $\ell\in\N$. Then, the grid points on level $\ell$ are
\[
x_{\ell,j}=j\times h_{\ell},\quad 0\leq j \leq 2^{\ell}.
\]
We can define the basis functions on level $\ell$ by
\[
\psi_{\ell,j}(x):\begin{cases}
1-\left|\frac{x}{h_{\ell}-j}\right|,&\quad\mbox{if }x\in[(j-1)h_{\ell},(j+1)h_{\ell}]\cap[0,1],\\
0,&\quad\mbox{otherwise.}\end{cases}
\]
Define the set on each level $\ell$ by
\[
B_{\ell}:=\left.\left\{ j\in\N:\begin{cases}
j=1,\dots,2^{\ell}-1,j\mbox{ is odd},&\quad\mbox{if }\ell>0\\
j=0,1,&\quad\mbox{if }\ell=0\end{cases}\right\}\right\}.
\]
The subspace of level $\ell$ is
\[
\W_{\ell}:=\spn\{\psi_{\ell,j}:j\in B_{\ell}\}.
\]
We denote $V_{\ell}$ as the subspace in $\Lt(I)$ up to level $\ell$, which is defined by the direct sum of subspaces 
\[
\V_{\ell}:=\bigoplus_{m\leq\ell}\W_{m}.
\]
Consequently, this yields the hierarchical structure of the subspace $\V_{\ell}$, namely 
\[
\V_{0}\subset\V_{1}\subset\cdots\subset\V_{\ell}\subset\V_{\ell+1}\cdots.
\]
Furthermore, the following hierarchical decomposition of the space $\Lt(I)$ holds
\[
\Lt(I)=\lim_{\ell\to\infty}\bigoplus_{m\leq \ell}\W_{m}.
\]
Note that one can derive the hierarchical decomposition of the space $\Lt(I^{d-1})$ for $d>1$ through the tensor product, which is denoted as $\V_{\ell}^{\otimes^{d-1}}$. We will use the subspace $\V_{\ell}$ to approximate the exact solution $u$ restricted on $\FH$.

Next, we present the approximation properties of the hierarchical space $\V_{\ell}$, which can be derived from standard $L^2$-projection error \cite{MR520174}, combining with the interpolation method.

\begin{prp}[Approximation properties of the hierarchical space $\V_{\ell}$]
\label{prop:haar}
Let $d=2,3$, $s\in (0,1]$ and let $\mathcal{I}_{\ell}: \Lt(I^{d-1})\to \V_{\ell}^{\otimes^{d-1}}$ be $\Lt$-projection for each level $\ell\geq 0$, then there holds,
\begin{align}\label{prop:approx-wavelets}
\|v-\mathcal{I}_{\ell}v\|_{\Lt(I^{d-1})}&\lesssim 2^{-s\ell}|v|_{\H^s(I^{d-1})}\quad\text{for all }v\in \H^s(I^{d-1}).
\end{align}
Here, $|\cdot|_{\H^s(I^{d-1})}$ denotes the Gagliardo seminorm in the fractional Sobolev space (or Slobodeskii space) $H^s(I^{d-1})$, given by
\begin{align*}
|v|^2_{\H^s(I^{d-1})}:=\int_{I^{d-1}}\int_{I^{d-1}}\frac{|v(x)-v(y)|^2}{|x-y|^{d-1+2s}}\mathrm{d}x\mathrm{d}y.
\end{align*}
The corresponding full norm is denoted as $\|\cdot\|_{\H^s(I^{d-1})}$.
\end{prp}
\subsection{Edge multiscale ansatz space}
Next, we introduce the construction of the edge multiscale ansatz space. 

First, we define the linear space over the inner boundary of each coarse neighborhood $\partial\omega_i\backslash\partial D$. Let the level parameter $\ell\in \mathbb{N}$ be fixed, and let $\Gamma_{i}^{j}$ with $j=1,\cdots,m$ be a partition of $\partial\omega_i\backslash\partial D$ with no mutual intersection, \ie, $\cup_{j=1}^{m}\overline{\Gamma_{i}^{j}}=\partial\omega_i\backslash\partial D$ and $\Gamma_{i}^{j}\cap \Gamma_{i}^{j'}=\emptyset$ if $j\neq j'$. Here, $m\in\mathbb{N}$ denotes the number of internal edges for $\omega_i$. Furthermore, we denote $\V_{i,\ell}^{j}\subset C(\partial\omega_{i}\backslash\partial D)$ as the linear space spanned by hierarchical bases up to level $\ell$ on each coarse edge $\Gamma_{i}^{j}$, which is continuous over $\partial \omega_i$ and vanishes on $\partial D$. The local edge space $\V_{i,\ell}$ defined over $\partial\omega_{i}$ is the smallest linear space having $\V_{i,\ell}^{j}$ as a subspace. Let $\{\psi_{\ell,i}^{j}\}_{j=1}^{n_i}$ and 
$\{\x_{\ell,i}^{j}\}_{j=1}^{n_i}$ be the nodal basis functions and the associated nodal points for $\V_{i,\ell}$. We can represent the local edge space $V_{i,\ell}$ by  
\begin{align}\label{eq:local-edge}
\V_{i,\ell} := \text{span} \left\{\psi_{\ell,i}^{j} : 1 \leq j \leq  n_i\right\}.
\end{align}
Next, we introduce the local multiscale space over each coarse neighborhood $\omega_i$, which is defined by 
\begin{align}\label{eq:local-multiscale}
\mathcal{L}^{-1}_{i}(\V_{i,\ell})
:= \text{span} \left\{\mathcal{L}^{-1}_{i}(\psi_{\ell,i}^{j}) : 1 \leq j \leq  n_i\right\}.
\end{align}
Here, $\mathcal{L}^{-1}_i (\psi_{\ell,i}^{j}):=v\in \H^{1}(\omega_i)$ is the solution to the following local problem, 
\begin{equation}
\label{eq:Li}
  \left\{ \begin{aligned}
          \mathcal{L}_{i} v&:=-\gd\cdot(\k\gd v)+\b\cdot \nabla v=0&& \mbox{in }\omega_{i},\\
          v&=\psi_{\ell,i}^{j}&& \mbox{on }\partial\omega_i\backslash\partial D\\
          v&=0&& \mbox{on }\partial\omega_i\cap\partial D.
  \end{aligned}\right.
\end{equation}
By this construction, the dimension of $\mathcal{L}^{-1}_{i}(\V_{i,\ell})$ equals $n_i$. In practice, a number of $n_i$ local problems \eqref{eq:Li} are solved by the standard numerical methods such as the FEMs with sufficient accuracy to obtain the local multiscale space $\mathcal{L}^{-1}_{i}(\V_{i,\ell})$, which can be solved in parallel and thus has low computational complexity. 

Finally, the edge multiscale ansatz space is defined by the Partition of Unity $\{\chi_{i}\}_{i=1}^N$,
\begin{align}\label{eq:global-multiscale}
\Vmslew := \text{span} \left\{\chi_{i}\mathcal{L}^{-1}_{i}(\psi_{\ell,i}^{j}) : \,  \, 1 \leq i \leq N \text{ and } 1 \leq j \leq  n_i\right\}\subset\V.
\end{align}

We summarize the construction of edge multiscale ansatz space in Algorithm~\ref{algorithm:wavelet}.
\begin{algorithm}[H]
\caption{Edge Multiscale ansatz space}
\label{algorithm:wavelet}
    \KwData{The level parameter $\ell\in \mathbb{N}$; coarse neighborhood $\omega_{i}$ and its four coarse edges $\Gamma_{i}^{j}$ with
    $j=1,\cdots, m$;
    the subspace $\V_{i,\ell}^{j}\subset \Lt(\Gamma_{i}^{j})$ up to level $\ell$ on each coarse edge $\Gamma_{i}^{j}$.
    }
    \KwResult{$\Vmslew$}
Construct the local edge space $V_{i,\ell}$ \eqref{eq:local-edge}\;    
Calculate the local multiscale space $\mathcal{L}^{-1}_{i} (\V_{i,\ell})$ \eqref{eq:local-multiscale}\;
Construct the global multiscale space $\Vmslew$ \eqref{eq:global-multiscale}.
\end{algorithm}
Once the local multiscale basis functions are identified, we solve for $\umsl(\cdot,t)\in \Vmslew$, such that
\begin{equation}\label{eq:ms-prob}
\begin{aligned}
(\pt\umsl,v)+A(\umsl,v)&= (R(\cdot,\umsl),v),\feac v\in\Vmslew,\\
\ums(\cdot,0)&=\mathcal{I}_{\ell}u_0.
\end{aligned}
\end{equation}
Here, $\mathcal{I}_{\ell}: \LtD\to \Vmslew$ denotes the $\Lt$-projection.

To obtain the convergence rates of \eqref{eq:ms-prob}, we introduce the global projection operator $\Pl$ of level $\ell:\V\to\Vmslew$. Since the edge multiscale ansatz space $\Vmslew$ is generated by the local multiscale space $\Liinv(\V_{i,\ell})$ using the partition of unity \eqref{eq:global-multiscale}, we only need to define the local interpolation operator $\Pil:C(\p\oi)\to\Linv(\V_{i,\ell})$. Let $\mathcal{I}_{i,\ell}:\Lt(\partial\oi)\to 
\V_{i,\ell}$ be the $\Lt$-projection, then we define $\Pil:\Lt(\p\oi)\to\Linv(\V_{i,\ell})$ by
\[
\Pil v:=\Liinv(\mathcal{I}_{i,\ell}v|_{\partial\oi}).
\]
Observe that any $v\in\V$ can be expressed by
\[
v = \sum_{i=1}^{N}\chi_{i}v|_{\oi}.
\]
Then the global interpolation $\Pl$ of level $\ell$ is defined by means of the local projection, 
\begin{equation}
\Pl v:=\sum_{i=1}^{N}\chi_{i}(\Pil v),\quad\feac v\in\V.
\end{equation}

\subsection{Error estimates}
\label{sec:loc-glo-split}
In this section, we are concerned with deriving the convergence rates of \eqref{eq:ms-prob} following the framework given in \cite{li2019convergence,fu2023wavelet}. 
The following result is a slight modification of Lemma 5.1 from \cite{fu2023wavelet}, which establishes \emph{a priori} estimates for the global bubble function $\uI$.

\begin{lem}
\label{lem:uiI}
Let $\uI$ be defined in \eqref{eq:def-uII}. Then we have
\begin{subequations}
\begin{align}\label{eq:bubble-est}
\|\uI\|_{\LtD}+H\|\gd\uI\|_{\LtkD}\leq H^{2}\left(\|u\|_{\HoD}
+\|\pt u\|_{\Lt(D)}\right).
\end{align}
\end{subequations}
\end{lem}
\begin{proof}
First, we multiply \eqref{eq:uiI-prob} by $\uiI$ and integrate over $\oi$ to obtain
\[
\int_{\oi}\k|\gd\uiI|^{2}\dx+\int_{\oi}\left(\b\cdot\gd\uiI\right)\uiI \dx=\int_{\oi}(R(\x,u)-\pt u)\uiI \dx,
\]
using the fact $\uiI|_{\p\oi}=0$, we reduce the expression above to
\[
\int_{\oi}\k|\gd\uiI|^{2}\dx=\int_{\oi}(R(\x,u)-\pt u)\uiI \dx.
\]
Combining the Schwarz inequality and the Friedrich's inequality, we arrive at
\[
\|\nabla\uiI\|_{\Hok(\oi)}^{2}\leq \frac{\sqrt{2}H}{\pi}\|R(\cdot,u)-\pt u\|_{\Lt(\oi)}\|\nabla\uiI\|_{\Hok(\oi)}.
\]
Since $R(\x,0)=0$, together with Assumption~\ref{asp:01} and Lemma \ref{lem:local-R}, we obtain the following
\[
\|\nabla\uiI\|_{\Hok(\oi)}\leq CH\left(\|u\|_{\Ho(\oi)}+\|\pt u\|_{\Lt(\oi)}\right).
\]
For the $\Lt$ estimate for the local function, we have again used the Friedrichs-Poincar\'e inequality.
\[
\|\uiI\|_{\Lt(\oi)}\preceq H\|\gd \uiI\|_{\Ltk(\oi)}\preceq H^{2}\left(\|u\|_{\Ho(\oi)}+\|\pt u\|_{\Lt(\oi)}\right).
\]
Finally, combining equation \eqref{eq:def-uII} and the overlap condition \eqref{eq:overlapcond} gives the desired result.
\end{proof}

Now, we introduce the approximation properties of the global interpolation operator $\Pl$.

\begin{lem}
\label{lem:Pl-interp}
Let $d=2,3$, and $s\in (0,1]$. Assume that the level parameter $\ell$ is non-negative. Let $u\in\V$ solve the problem \eqref{eq:main-prob} and define the global Harmonic extension $\uII$ in \eqref{eq:def-uII}. Then, there holds
\begin{subequations}
\begin{align}
\|\uII-\Pl\uII\|_{\Lt_{\widetilde{\kappa}}(D)}&\preceq \eta(H, \ell,s)H\|u\|_{\H^{s+1/2}(D)},\label{eq1:leminterpuII}\\
\|\gd(\uII-\Pl\uII)\|_{\LtkD}&\preceq \eta(H, \ell,s)\|u\|_{\H^{s+1/2}(D)}.\label{eq2:leminterpuII}
\end{align}
\end{subequations}
Here, 
\begin{align}\label{eq:etaHl}
\eta(H, \ell,s):= \|\kappa\|_{\Linf(\FH)}^{1/2} 2^{-s\ell}H^{s-1/2}.
\end{align}
\end{lem}
\begin{proof}
First, the property of the partition of unity  of $\{\chi_{i}\}_{i=1}^{N}$ along with expression \eqref{eq:def-uII} yields
\[
\uII-\Pl\uII=\sum_{i=1}^{N}\chi_{i}(\uiII-\Pil\uiII).
\]
Using the rescaling argument and Proposition \ref{prop:haar}, we obtain
\[
\|\uiII-\Pil\uiII\|_{\Lt(\p\oi)}\leq C2^{-s\ell}H^{s}|\uiII|_{\H^{s}(\p\oi)}.
\]
By definition, $\uiII=u$ on $\partial\omega_i$, together with the trace inequality, this implies
 \begin{align}\label{eq:edge-app}
\|\uiII-\Pil\uiII\|_{\Lt(\p\oi)}\leq C2^{-s\ell}H^{s}\|u\|_{\H^{s+1/2}(\oi)}.
\end{align}
Following the rescaling argument from \cite[Theorem A.1, pp.~615]{li2019convergence}, we can prove
\begin{align}
\|\uiII-\Pil\uiII\|_{\Lt_{\widetilde{\kappa}}(\oi)}
&\leq C H^{1/2}\|\uiII-\Pil\uiII\|_{\Lt_{\kappa}(\partial\oi)}\nonumber\\
&\leq C H^{1/2}\|\kappa\|_{\Linf(\partial\oi)}^{1/2}\|\uiII-\Pil\uiII\|_{\Lt(\partial\oi)}\nonumber.
\end{align} 
This, together with \eqref{eq:edge-app}, leads to 
\begin{align}\label{eq:loc-l2}
\|\uiII-\Pil\uiII\|_{\Lt_{\widetilde{\kappa}}(\oi)}
&\leq C 2^{-s\ell}H^{s+1/2}\|\kappa\|_{\Linf(\partial\oi)}^{1/2}\|u\|_{\H^{s+1/2}(\oi)}.
\end{align} 
By a standard Caccioppoli-type inequality, see, e.g., \citet[Lemma 4.4]{li2019convergence}, we derive 
\begin{align}\label{eq:loc-eng}
\|\chi_i\nabla(\uiII-\Pil\uiII)\|_{\Ltk(\oi)}
&\leq C 2^{-s\ell}H^{s-1/2}\|\kappa\|_{\Linf(\partial\oi)}^{1/2}\|u\|_{\H^{s+1/2}(\oi)}.
\end{align} 

Invoking the overlapping condition \eqref{eq:overlapcond} and \eqref{eq:loc-l2}, leads to
\begin{align*}
\|\uII-\Pl\uII\|_{\Lt_{\widetilde{\kappa}}(D)}&\leq \sqrt{\Cov}C_{\infty} \left(
\sum_{i=1}^{N}\|\uiII-\Pil\uiII\|^{2}_{\Lt_{\widetilde{\kappa}}(\oi)}\right)^{1/2}\\
&\preceq 2^{-s\ell}H^{s-1/2}\|\kappa\|_{\Linf(\mathcal{F}_H)}^{1/2}\|u\|_{\H^{s+1/2}(D)}.
\end{align*}
This proves \eqref{eq1:leminterpuII}. 

Next, applying the overlapping condition \eqref{eq:overlapcond} again, we obtain
\begin{equation}
\label{eq1:main-GuiII}
\begin{split}
\|\nabla(\uiII-\Pil\uiII)\|_{\Ltk(D)}&\leq \sqrt{2\Cov}\left(\sum_{i=1}^{N}C_{\mathrm{G}}^{2}H^{-2}\|\uiII-\Pil\uiII\|^{2}_{\Lt_{\widetilde{\kappa}}(\oi)}\right.\\
&\quad+\|\chi_{i}\gd (\uiII-\Pil\uiII)\|^{2}_{\Ltk(\oi)}\Bigg)^{1/2}.
\end{split}
\end{equation}
Combining with \eqref{eq:loc-l2} and \eqref{eq:loc-eng} reveals the second assertion. 
This completes the proof.
\end{proof}
Let $\mathcal{R}_{\ell}$ be the elliptic projection operator in the multiscale space $\Vmslew$, \ie,
\begin{align}\label{eq:riesz}
\feac v\in \V \text{ and }\wmsl\in \Vmslew:A(v-\mathcal{R}_{\ell}v,\wmsl)=0.
\end{align}
\begin{lem}\label{lem:riesz}
Let $\ell$ be a non-negative integer and $s\in(0,1]$. Let $u$ be the solution to problem \eqref{eq:main-prob}. Then there holds
\begin{equation}
\begin{split}
\|\nabla({u-\mathcal{R}_{\ell}u})\|_{\LtkD}+\|{u-\mathcal{R}_{\ell}u}\|_{\LtD}&\lesssim \eta(H,\ell,s)
\|u\|_{\H^{s+1/2}(D)}\\
&\quad+H\left(\|u\|_{\HoD}+\|\pt u\|_{\Lt(D)}\right).\label{eq:glo-energy2}
\end{split}
\end{equation}

\end{lem}
\begin{proof}
We obtain from the global decomposition \eqref{eq:decomp-u}, \eqref{eq:coercive} and \eqref{eq:riesz},
\[
\|\nabla({u-\mathcal{R}_{\ell}u})\|_{\LtkD}^2= A(u-\mathcal{R}_{\ell}u,u-\mathcal{R}_{\ell}u)\leq A(u-\mathcal{R}_{\ell}u,u^{\roma}+e^{\romb}).
\]
Here, $e^{\romb}:=u^{\romb}-\mathcal{P}_{\ell}u^{\romb}$. 
The definition of the bilinear form $A(\cdot,\cdot)$, together with the Cauchy-Schwarz inequality, further leads to
\begin{align*}
\|\nabla(u-\mathcal{R}_{\ell}u)\|_{\LtkD}^2&=\int_D\kappa\nabla(u-\mathcal{R}_{\ell}u)\cdot\nabla(u^{\roma}+e^{\romb})\dx
+\int_{D}\b\cdot \nabla(u-\mathcal{R}_{\ell}u)(u^{\roma}+e^{\romb})\dx\\
&\leq\|\nabla({u-\mathcal{R}_{\ell}u})\|_{\Ltk(D)}\|\nabla(u^{\roma}+e^{\romb})\|_{\Ltk(D)}\\
&\quad+\|\b\|_{\LinfD}\normL{\nabla(u-\mathcal{R}_{\ell}u)}{\LtD}\normL{u^{\roma}+e^{\romb}}{\LtD}.
\end{align*}
Since $\kappa\geq 1$, we derive
\begin{align*}
\|\nabla({u-\mathcal{R}_{\ell}u})\|_{\Ltk(D)}\leq \|\nabla(u^{\roma}+e^{\romb})\|_{\Ltk(D)}+
\normL{u^{\roma}+e^{\romb}}{\LtD}.
\end{align*}
Finally, an application of Lemma \ref{lem:uiI}, Lemma \ref{lem:Pl-interp}, and Friedrich's inequality yields the desired assertion.
\end{proof}
Since $u_0\in\dot{\H}^2(D)$, similar argument as the proof to Lemma \ref{lem:riesz} leads to
\begin{equation}\label{eq:glo-energy-int}
\|\nabla({u_0-\mathcal{R}_{\ell}u_0})\|_{\Ltk(D)}+\normL{u_0-\mathcal{R}_{\ell}u_0}{\LtD}\lesssim \eta(H,\ell,s)
\|u_0\|_{\H^{s+1/2}(D)}+H\|\mathcal{L}u_0\|_{\LtD}.
\end{equation}
Finally, We present an estimate for the semi-discrete formulation given in \eqref{eq:ms-prob}, which follows from \citealp[Theorem 3.1]{Larsson92} or \cite[Theorem 14.2, pp.~246]{thomee2006galerkin}, together with Lemma \ref{lem:riesz} and \eqref{eq:glo-energy-int}.
\begin{thm}
\label{thm:theta-estimate}
Let $\ell$ be a non-negative integer, $s\in (0,1]$ and $u_0\in \dot{\H}^{2}(D)$. Let $u$ and $\umsl$ be the solutions to Problem \eqref{eq:main-prob} and Problem \eqref{eq:ms-prob}, respectively. There holds
\begin{align*}
\|\nabla({u(t)-\umsl(t)})\|_{\Ltk(D)}
&+\|u(t)-\umsl(t)\|_{\LtD}\\
&\preceq \eta(H,\ell,s)
\left(\|u(t)\|_{H^{s+1/2}(D)}+\|u_0\|_{\H^{s+1/2}(D)}\right)\\
&\quad+H\left(\|u(t)\|_{\HoD}+\|\pt u(t)\|_{\Lt(D)}+\|\mathcal{L}u_0\|_{\LtD}\right),
\end{align*}
for all $t\in(0,T]$.
\end{thm}

\section{Full discretization}
\label{sec:full-discretization}

In this section, we introduce the fully discrete scheme of \eqref{eq:weak-prob}  based on exponential integration techniques \cite{hochbruck1998exponential,hochbruck2005explicit,hochbruck2010exponential}. 
The effectiveness of these approaches is attributed to the computation of the matrix-valued exponential integrator. Their prominence has increased in the past decade due to advancements in numerical linear algebra and efficient algorithms for computing these functions \citep{caliari2024accelerating,caliari2024efficient}. Consequently, these methods have become significantly more competitive, offering an alternative to traditional temporal discretizations.

Let $0=t^{0}<t^{1}<\cdots<t^{N_{t}-1}<t^{N_{t}}=T$ be a uniform partition of the interval $[0,T]$, with the time-step size given by $\dt=\max_{n=1,\dots,N_{t}}\{t^{n}-t^{n-1}\}$, where $N_{t}$ is a positive integer. We recall exponential integration techniques relevant to the proposed multiscale finite element method. 

First, we define the discrete solution $\umsl(t)$ evolving in time. Since that $A(\cdot,\cdot)$ is bounded and coercive over $\Vmslew\times\Vmslew$, then Riesz's representation Theorem implies the existence of a bounded linear operator $\Lms:\Vmslew\to\Vmslew'$, satisfying
\begin{align}\label{eq:lms}
A(\umsl,\vmsl)=\langle\Lms\umsl,\vmsl\rangle\quad\feac\vmsl\in\Vmslew.
\end{align}
Here, $\langle \cdot,\cdot\rangle$ denote the duality paring.
Then we can rewrite \eqref{eq:weak-prob} in a semidiscrete formulation in $\Vmslew$ 
\begin{equation}
\label{eq:Lh-fem}
\begin{aligned}
\pt\umsl+\Lms\umsl &=\Pms R(\cdot,\umsl)&\quad t\in(0,T],\\
\umsl(0) &= \Pms u_{0}&
\end{aligned}
\end{equation}
where $\Pms: \LtD\to \Vmslew$ denotes $\Lt$-projection.

Following \cite{henry1981geometric}, the $\V$-elliptic property of $\Lms$ implies that $-\Lms$ is a sectorial on $\LtD$ (uniformly in $h$), \ie, there exist $C>0$ and $\theta \in(\tfrac{1}{2}\pi,\pi)$ such that
\[
\|(\lambda \Im +\Lms)^{-1}\|_{L(\LtD,\LtD)}\leq \frac{C}{|\lambda|},\quad \lambda\in S_{\theta},
\]
where $S_{\theta}:=\{\lambda\in\Cx:\lambda=\rho e^{i\phi},\rho>0,0\leq\phi\leq \theta\}$. The discrete operator $-\Lms$ is the infinitesimal of the exponential operator (or semigroup) $e^{-t\Lms}$ on $\Vmslew$ such that
\[
e^{-t\Lms}:=\frac{1}{2\pi}\int_{\Gamma}e^{t\lambda}(\lambda\Im+\Lms)^{-1}d\lambda,\quad t>0,
\]
where $\Gamma$ represents a path surrounding the spectrum of $-\Lms$. According to Duhamel's principle, we ensure the existence and uniqueness of the solution, which is given by the integral form
\begin{equation}
\label{eq:int-form}
\umsl(t^{n+1}) = e^{-\dt\Lms}\umsl(t^{n})+\int_{0}^{\dt}e^{-(\dt-s)\Lms}\Pms R\left(\cdot,\umsl(t^{n}+s)\right)ds.
\end{equation}
Approximating $\umsl(t^n+s)$ by $\umsl(t^n)$ for $s\in [0,\dt]$ \citep{hochbruck2010exponential}, we can obtain a fully discrete multiscale method for Problem \eqref{eq:ms-prob} as follows: for $n=0,\dots, N_{t}-1$,
\begin{equation*}
\unpmsl = e^{-\dt\Lms}\unmsl+\Lms^{-1}(I-e^{-\dt\Lms})\Pms R(\cdot,\unmsl).
\end{equation*}
Together with the fact $e^{-\dt\Lms}=-\dt\Lms\phi_{1}(-\dt\Lms)+\Im$ with $\phi_1(z)=\frac{e^z-I}{z}$, 
this implies
\begin{equation}
\label{eq:RK1-1}
\unpmsl = \unmsl+\dt\phi_{1}(-\dt\Lms)\left(\Pms R(\cdot,\unmsl)-\Lms\unmsl\right),
\end{equation}
which is the so-called exponential Euler scheme. Note that this scheme involves huge computational complexity when applied to \eqref{eq:main-prob} directly using standard numerical methods such as the finite element method, due to the existence of heterogeneous coefficient $\kappa$ and the evaluation of the exponential-like matrix function $\phi_1(\cdot)$. In sharp contrast, our proposed edge multiscale method \eqref{eq:ms-prob} offers a significant advantage. It reduces the size of the exponential-like matrix function $\phi_1(\cdot)$, thereby making exponential integration feasible even for such challenging multiscale problems.

The following result will be mainly used in this work.
\begin{prp}[Properties of the semigroup \citealp{henry1981geometric}]
\label{prp:prop-semigroup}
Let $\alpha\geq 0$ and any given parameter $0\leq \gamma\leq 1$, then there exists a constant $C>0$ such that
\begin{align*}
\|\Lms^{\alpha}e^{-t\Lms}\|_{\LtD}\leq Ct^{-\alpha},\quad\mbox{for }t>0,\\
\|\Lms^{-\gamma}(I-e^{-t\Lms})\|_{\LtD}\leq Ct^{\gamma},\quad\mbox{for }t\geq0.
\end{align*}
\end{prp}
\subsection{Error estimates}

This subsection concerns the convergence study of fully discrete scheme \eqref{eq:RK1-1}. 

\begin{thm}
\label{thm:main-result}
Let $\ell$ be a non-negative integer and $u_0\in \dot{\H}^{2}(D)$. Let $u\in \V$ and $\unmsl$ be the solution to Problems \eqref{eq:main-prob} and \eqref{eq:RK1-1}, respectively. When the stepping size $\Delta t$ is sufficiently small, there holds
\begin{align*}
\|u(t^{n})-\unmsl\|_{\LtD}
 &\preceq \eta(H,\ell,1)
+H+\Delta t.\\
\|\nabla({u(t^{n})-\unmsl})\|_{\Ltk(D)}
 &\preceq \eta(H,\ell,1)
+H+\Delta t \left(t^{n}\right)^{-1/2}.
\end{align*}
Here, the hidden constant may depend on $R,\b$, and $T$.
\end{thm}
\begin{proof}
The proof to the first estimate is similar to \cite[proof to Theorem 4.1]{tambue2016exponential}. To obtain the second estimate, 
an application of the triangle inequality yields
\begin{align*}
\|\nabla({u(t^{n})-\unmsl})\|_{\Ltk(D)}&\leq \|\gd(u(t^{n})-\umsl(t^{n}))\|_{\Ltk(D)}+\|\gd(\umsl(t^{n})-\unmsl)\|_{\Ltk(D)}\\
&:= I_{1} + I_{2}.    
\end{align*}
From Theorem \ref{thm:theta-estimate} and Lemma \ref{lem:02}, we get an upper bound for the first term $I_{1}$. Thus, our focus is on the second term $I_{2}$. Note that the integral form of the semidiscrete solution $\umsl$ \eqref{eq:int-form} and its approximation $\unmsl$ \eqref{eq:RK1-1} indicate,
\begin{align*}
\umsl(t^{n})&= e^{-\dt\Lms}\umsl(0)+\sum_{k=0}^{n-1}
\int_{t^{k}}^{t^{k+1}}e^{-(t^{n}-s)\Lms}\Pms R(\cdot,\umsl(s))\mathrm{d}s,\\
\unmsl&= e^{-\dt\Lms}\umsl(0)+\sum_{k=0}^{n-1}\int_{t^{k}}^{t^{k+1}}e^{-(t^{n}-s)\Lms}\Pms R(\cdot,\ukmsl)\mathrm{d}s.
\end{align*}
Hence, subtracting the second equation from the first one, we obtain the temporal discretization error, 
\[
\varepsilon_{\ms,\ell}^{n}:=\umsl(t^{n})-\unmsl= \sum_{k=0}^{n-1}\int_{t^{k}}^{t^{k+1}}e^{-(t^{n}-s)\Lms}\Pms \left(R(\cdot,\umsl(s))-R(\cdot,\ukmsl)\right)\mathrm{d}s.
\]
Then, together with the triangle inequality and the equivalence of norms by \eqref{eq:coercive} and \eqref{eq:lms}, 
\begin{align*}
\|\gd\vmsl\|_{\Ltk(D)}= \|\Lms^{1/2}\vmsl\|_{\LtD},
\quad\feac \vmsl\in \Vmslew,
\end{align*}
this leads to 
\begin{align*}
\|\gd\varepsilon_{\ms,\ell}^{n}\|_{\LtkD}
&\leq \sum_{k=0}^{n-1}\int_{t^{k}}^{t^{k+1}}\left\|\nabla e^{-(t^{n}-s)\Lms}\Pms \left(R(\cdot,\umsl(s))-R(\cdot,\ukmsl)\right)\right\|_{\Ltk(D)}\mathrm{d}s\\
&=\sum_{k=0}^{n-1}\int_{t^{k}}^{t^{k+1}}\left\|e^{-(t^{n}-s)\Lms}\Lms^{1/2}\Pms \left(R(\cdot,\umsl(s))-R(\cdot,\ukmsl)\right)\right\|_{\LtD}\mathrm{d}s.
\end{align*}
Since $\Pms$ is the $L^2$-projection, we get therefore, 
\begin{align*}
\|\gd\varepsilon_{\ms,\ell}^{n}\|_{\LtkD}\leq 
\sum_{k=0}^{n-1}\int_{t^{k}}^{t^{k+1}}\left\|e^{-(t^{n}-s)\Lms}\Lms^{1/2}\right\|_{\LtD} \left\|R(\cdot,\umsl(s))-R(\cdot,\ukmsl)\right\|_{\LtD}\mathrm{d}s.
\end{align*}
The remaining proof is similar to \cite[proof to Theorem 4.1]{tambue2016exponential} using Proposition \ref{prp:prop-semigroup} and Lemma \ref{lem:02}.
\end{proof}
\section{Numerical tests}
\label{sec:numerical-exp}
In this section, we present several numerical tests to demonstrate the accuracy of Scheme \eqref{eq:RK1-1}. 
To obtain a reference solution with sufficient accuracy, we set the fine-scale mesh size of $h=2^{-9}$ for $d=2$ and $h=2^{-6}$ for $d=3$. Moreover, we take $h=2^{-8}$ in Example \ref{example1s} of 2-d two-component coupled systems. The backward Euler method is utilized for temporal discretization for the reference solution. A Homogeneous Dirichlet boundary condition is imposed for all tests.

The functions $\chi_{i}$ are the standard multiscale basis functions on each coarse element $K\in\TH$ defined by
\begin{equation}
\label{eq:pu-problem}
\begin{cases}
-\gd\cdot(\k\gd \chi_{i})=0,&\quad\mbox{in }K,\\
\chi_{i}=g_{i},&\quad\mbox{on }\p K,
\end{cases}
\end{equation}
where $g_{i}$ is affine over $\p K$ with $g_{i}(\x_{j})=\delta_{ij}$ for all $i,j=1,\dots,N$. Recall that $\{\x_{j}\}_{j=1}^{N}$ are the set of coarse nodes on $\TH$. 

To measure the approximation accuracy, we consider the following notation for the relative errors in $\Lt$- and $\H_{\kappa}^1$-norm and the numerical convergence rate $\text{CR}$:
\[
\varepsilon_{0}=\frac{\|\uh^{N}-\umsl^{N}\|_{\LtD}}{\|\uh^{N}\|_{\LtD}},\quad \varepsilon_{1}=\frac{\|\nabla(\uh^{N}-\umsl^{N})\|_{\LtkD}}{\|\nabla\uh^{N}\|_{\LtkD}},\quad \text{CR} = \frac{|\ln\varepsilon_{\star}^{H}-\ln\varepsilon_{\star}^{H/2}|}{\ln 2},
\]
where $\uh^{N}$ and $\ums^{N}$ denotes the reference and multiscale solution at final time $t=T$ respectively. $\varepsilon_{\star}^{H}$ denotes the relative error with coarse mesh size $H$, for $\star\in\{0,1\}$. Consequently, $\operatorname{CR}$ measures the convergence rate in $H$. In addition, we shall evaluate \eqref{eq:RK1-1} via Pad\'e approximations implemented via \texttt{EXPINT} package in \cite{berland2007expint}.

\begin{example}
\label{example1}
First, we consider the 2-d convective Allen-Cahn problem,
\begin{equation}
\label{eq:example1}
\pt u -\dive(\k\gd u)+\b\cdot\gd u = \frac{1}{\epsilon^{2}}(u-u^3),\quad\mbox{in }D\times(0,T],
\end{equation}
where $\epsilon=0.1$ measures the interface thickness, $D=[-1,1]^2$, and $T=1$ and time step size of $\dt = 2^{-10}$. The diffusion coefficient $\k=1$ and the incompressible velocity field is $\b:=\e^{-t}(\cos(2\pi y),\cos(2\pi x))^{T}$. The initial data is $u_{0}:=\sin(2\pi x)\sin(2\pi y)$.


\begin{table}[!t]
\centering
\begin{tabular}{ccccccc}
\toprule 
\multicolumn{7}{c}{Multiscale approximation at $T=1$}\tabularnewline
\midrule 
\multirow{2}{*}{$H$} & \multicolumn{2}{c}{$\ell=0$} & \multicolumn{2}{c}{$\ell=1$} & \multicolumn{2}{c}{$\ell=2$}\tabularnewline
\cmidrule{2-7} \cmidrule{3-7} \cmidrule{4-7} \cmidrule{5-7} \cmidrule{6-7} \cmidrule{7-7} 
 & $\varepsilon_{0}$ & $\text{CR}$ & $\varepsilon_{0}$ & $\text{CR}$ & $\varepsilon_{0}$ & $\text{CR}$\tabularnewline
\midrule 
$2^{-1}$ & 1.0638E+00 & -- & 2.4791E-02 & -- & 1.5226E-02 & -- \tabularnewline
\midrule 
$2^{-2}$ & 4.4884E-02 & 4.5669  & 9.2118E-03 & 1.4283 & 4.0592E-03 & 1.9073 \tabularnewline
\midrule 
$2^{-3}$ & 4.3052E-03 & 3.3821 & 1.2737E-03 & 2.8544 & 6.5153E-04 & 2.6393 \tabularnewline
\midrule 
$2^{-4}$ & 1.5135E-03 & 1.5081 & 3.0049E-04 & 2.0836 & 2.7358E-05 & 4.5738 \tabularnewline
\midrule 
$2^{-5}$ & 2.3453E-04 & 2.6901 & 4.8751E-05 & 2.6238 & 4.6474E-06 & 2.5575 \tabularnewline
\bottomrule
\end{tabular}
\caption{Relative errors and convergence rates for Example \ref{example1} at $T=1$ with varying parameters $(H,\ell)$.}
\label{tab:example1}
\end{table}

Table~\ref{tab:example1} reports the numerical results of \eqref{eq:RK1-1} with level parameter $\ell=0,1,2$ and varying coarse grid size $H=2^{-j}$, where $j\in\{1,\dots,5\}$. We observe that its accuracy improves as the coarse grid size $H$ decreases and the level parameter $\ell$ increases. In addition, all levels have a rate of convergence close to $1$ with time-stepping fixed, which coincides with the theoretical results given in Theorem \ref{thm:main-result}. Figure~\ref{fig:example1} depicts the snapshots of the reference and multiscale solution with $\ell=1$ and $H=2^{-4}$ and different time steps at $t=0,0.0098,0.125$ and $0.75$.

In Table~\ref{tab:timeexample1}, we present the relative errors and their corresponding convergence rates for $\ell=1$, where the temporal convergence rates are much better than expected (higher than 1) from theoretical results given in Theorem~\ref{thm:main-result}. We mention that for a fixed coarse size $H$, the convergence rate in time degrades as the spatial error gradually dominates the total error.
\begin{table}
\centering
\begin{tabular}{cccccc}
\toprule 
\multicolumn{6}{c}{Multiscale approximation at $T=0.01$}\tabularnewline
\midrule 
\multirow{1}{*}{$H$} & $\dt$ & $\varepsilon_{0}$ & $\text{CR}$ & $\varepsilon_{1}$ & $\text{CR}$\tabularnewline
\midrule 
$2^{-2}$ & $2^{-5}$ & 2.0231E--02  & -- & 2.0231E--02 & --\tabularnewline
\midrule 
$2^{-3}$ & $2^{-6}$ & 1.3950E--02 & 0.5363  & 1.3950E--02 & 6.4657E--01 \tabularnewline
\midrule 
$2^{-4}$ & $2^{-7}$ & 5.8310E--03 & 1.2584  & 5.8310E--03  & 1.7156E+00 \tabularnewline
\midrule 
$2^{-5}$ & $2^{-8}$ & 7.5210E--04 & 2.9547 & 7.5210E--04 & 2.1164E+00\tabularnewline
\bottomrule
\end{tabular}
\caption{\label{tab:timeexample1} Relative errors and convergence rates in time, for Example~\ref{example1} at $T=0.01$ decreasing  simultaneously the time step size $\dt$ and the coarse size $H$ with $\ell=1$.} 
\end{table}

\begin{figure}[!t]
\centering
\subfloat[\label{exp1:umsl}Multiscale solution.]{{\includegraphics[width=0.24\textwidth]{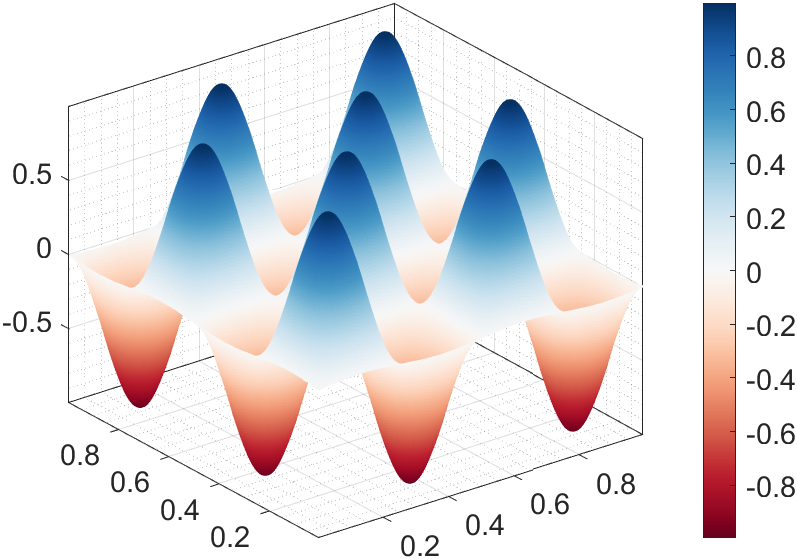}}~{\includegraphics[width=0.24\textwidth]{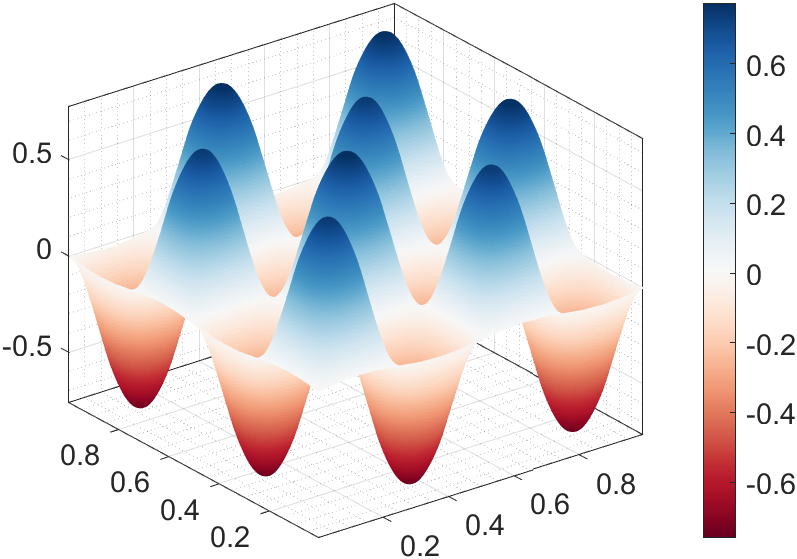}}~{\includegraphics[width=0.24\textwidth]{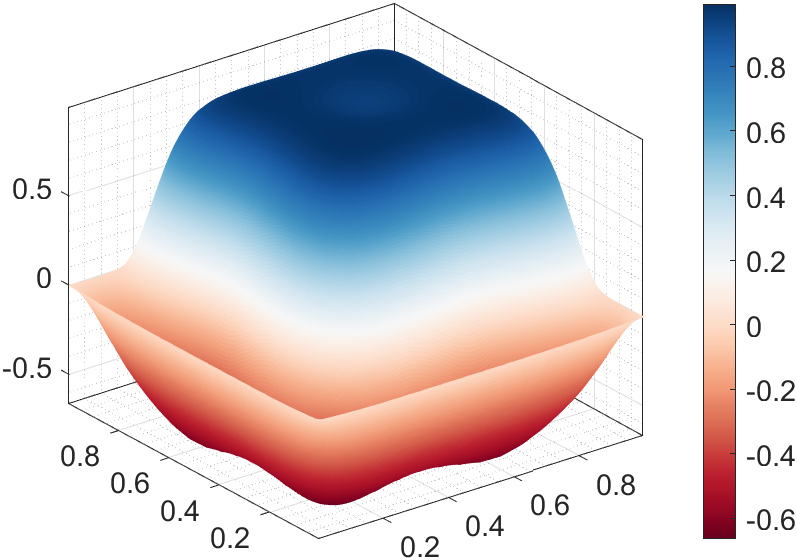}}~
~{\includegraphics[width=0.24\textwidth]{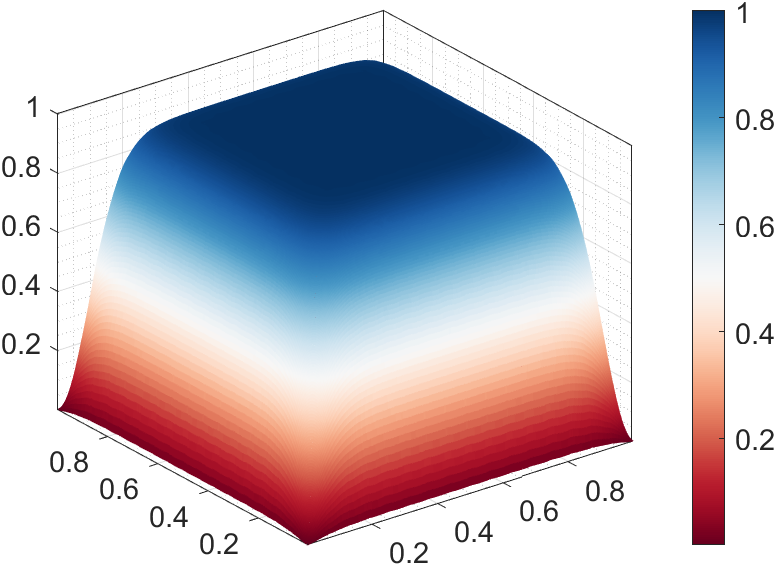}}}\\
\subfloat[exa1rf][Reference solution.]{{\includegraphics[width=0.24\textwidth]{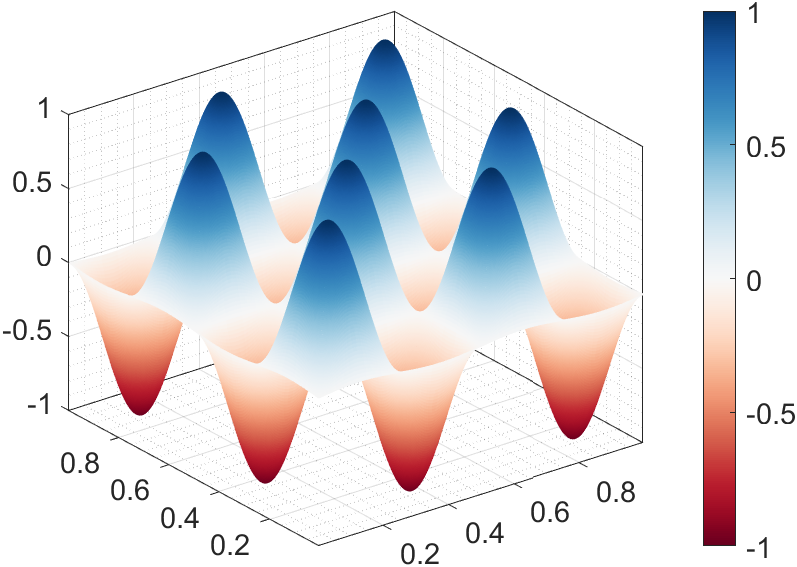}}
~{\includegraphics[width=0.24\textwidth]{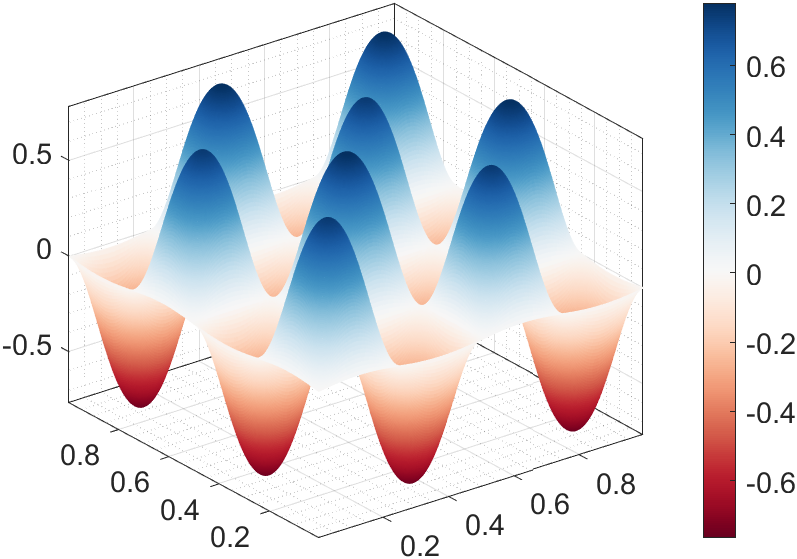}}~{\includegraphics[width=0.24\textwidth]{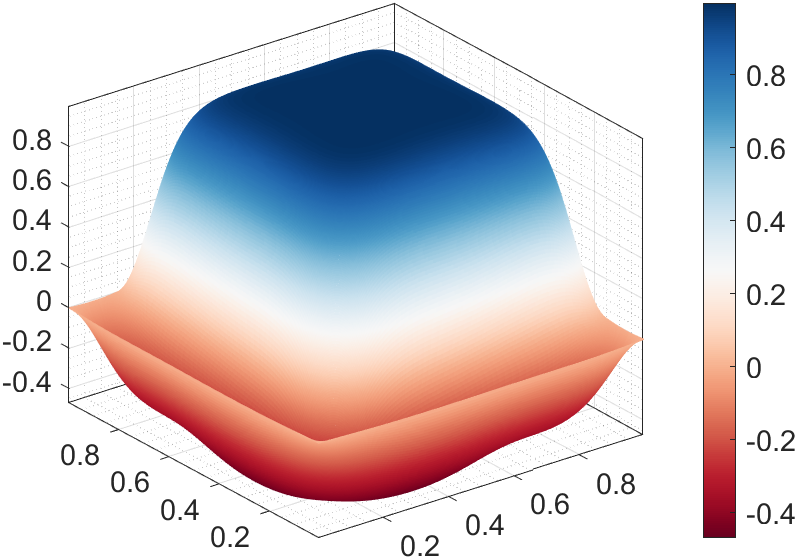}}~{\includegraphics[width=0.24\textwidth]{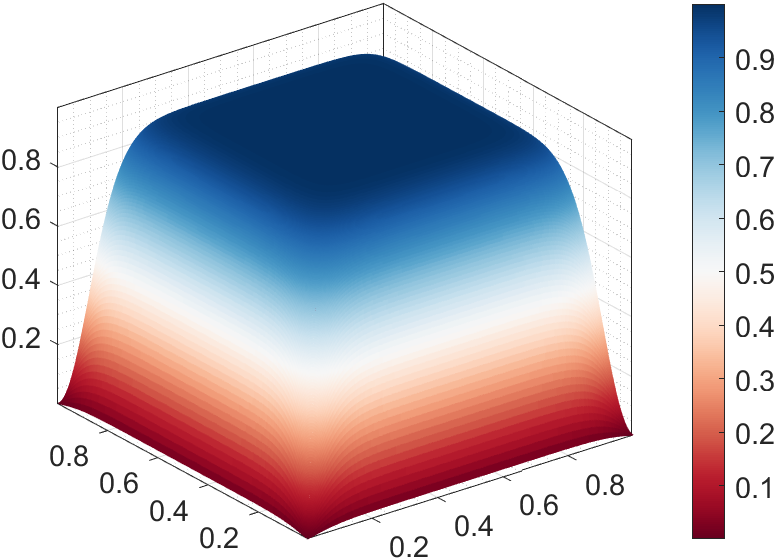}}}
\caption{Solutions of Example~\ref{example1} at $t=0,0.098,0.125$ and $0.750$ (a) Multiscale solution with $H=2^{-4}$ and $\ell=1$, and (b) the reference solution.}
\label{fig:example1}
\end{figure}

We present in Figure~\ref{fig:E-example1} the corresponding evolution of the classic free energy ${\cal E}$ (see, for instance, \citealp{Bartels2015numerical}) defined by
\[
{\cal E}[u] := \int_{D}\left(\frac{1}{2}|\gd u|^{2}+\frac{1}{4\epsilon^{2}}(u^{2}-1)^{2}\right)\dx,
\]
and $\max$-norm of the multiscale solution using $H=2^{-4}$ and $\ell=1$. We observe monotone decay of the discrete energy over time, and the discrete maximum principle is well-preserved.
\begin{figure}[!h]
\centering
\begin{tikzpicture}
\begin{axis}[
grid=minor,
scale = 1.2,
xlabel={\small{$t$}}, 
ylabel={\small{${\cal E}[\umslew]$}},
]
\addplot[DarkBlue] table [x=T,y=EN] {EMei2d1_Hierarchical.dat};
\end{axis}
\end{tikzpicture}~
\begin{tikzpicture}
\begin{axis}[
grid=minor,
scale = 1.2,
xlabel={\small{$t$}}, 
ylabel={\small{$\|\umslew\|_{\LinfD}$}},
]
\addplot[Red] table [x=T,y=MAX] {EMei2d1_Hierarchical.dat};
\end{axis}
\end{tikzpicture}
\caption{\label{fig:E-example1} Time-dependent normalized discrete total energy (left) and maximum values of the multiscale solution $\umslew$ (right) over time evolution with $\dt=2^{-10}$ in Example~\ref{example1}.}
\end{figure}
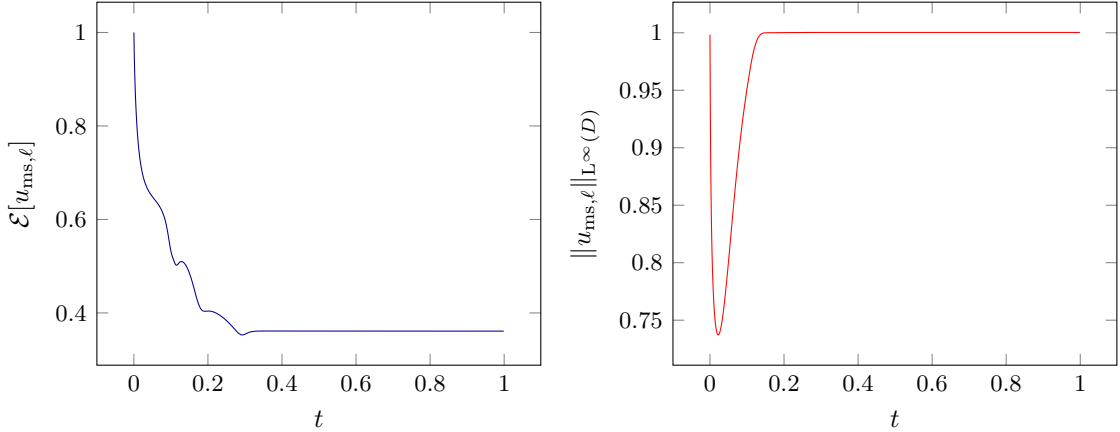
\end{example}

\begin{example}
\label{example2}
Next, we investigate the performance of the proposed method for the convective Allen-Cahn problem with heterogeneous coefficient,
\begin{equation}
\label{eq:example2}
\pt u -\dive(\k_{2}\gd u)+\b\cdot\gd u = \frac{1}{\epsilon^{2}}(u-u^3),\quad\mbox{in }D\times(0,T],
\end{equation}
where $D=[0,1]^2$, $T=0.016$, and $\epsilon=0.05$. The heterogeneous coefficient $\k_{2}$ is depicted in Figure~\ref{fig:perm-example2}, which takes value $1$ in the background (gray region) and $10^{4}$ in other parts (red region).
\begin{figure}[!h]
\centering
\subfloat[$\k_{1}$]{{\includegraphics[width=0.28\textwidth]{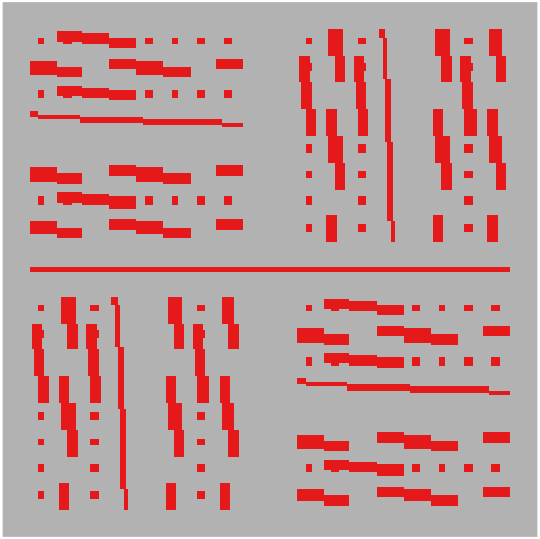}}\label{fig:perm-example2s}}~
\subfloat[$\k_{2}$]{{\includegraphics[width=0.28\textwidth]{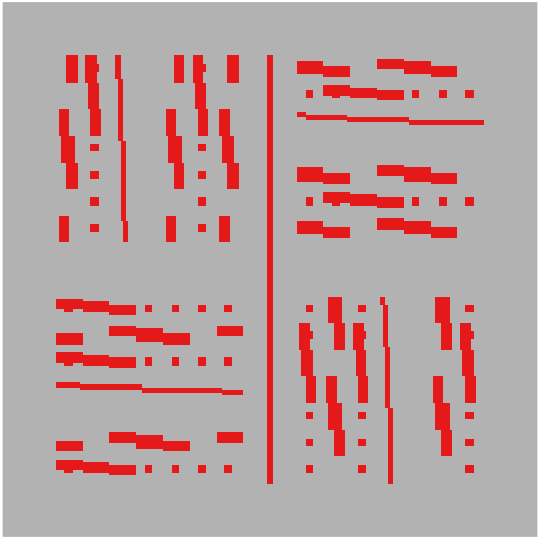}}\label{fig:perm-example2}}~
\subfloat[$\k_{3}$]{{\includegraphics[width=0.28\textwidth]{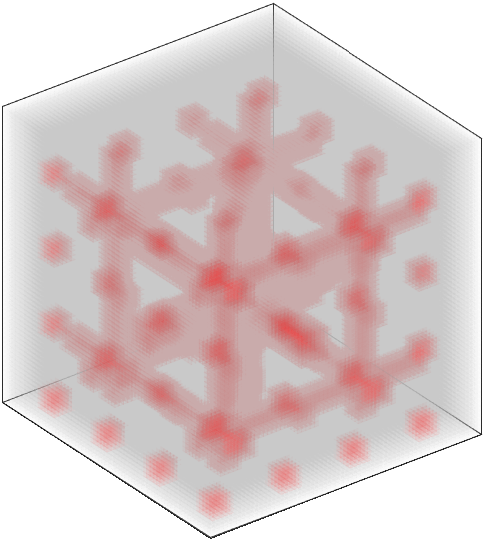}}\label{fig:perm3d-ex3}}
\caption{\label{fig:perm-example2-2s} Permeability fields (color online).}
\end{figure}

We set the velocity term as
\[
\b = \alpha(\cos(k\pi y)\sin(k\pi x),-\cos(k\pi x)\sin(k\pi y))^{T},
\]
where $\alpha=2$ and $k=24$. The initial data is $u_{0}:=0.1\sin(\pi x)\sin(\pi y)$. 

Table~\ref{tab:example2} exhibits convergence history of \eqref{eq:RK1-1} with varying coarse grid size $H=2^{-j}$, where $j\in\{1,\dots,5\}$ and the level parameter of $\ell=0,1,2$. We observe convergence as the coarse grid size $H$ decreases and level parameter $\ell$ increases, which agrees with Example \ref{eq:example1}. Figure~\ref{fig:example2} depicts the snapshots of the multiscale solution with $\ell=2$ and $H=2^{-4}$, and reference solutions at $t=0$, $1.5625\text{E-}04$, $0.004$ and $0.012$.

\begin{table}[!h]
\centering
\begin{tabular}{ccccccc}
\toprule 
\multicolumn{7}{c}{Multiscale approximation at $T=0.016$}\tabularnewline
\midrule 
\multirow{2}{*}{$H$} & \multicolumn{2}{c}{$\ell=0$} & \multicolumn{2}{c}{$\ell=1$} & \multicolumn{2}{c}{$\ell=2$}\tabularnewline
\cmidrule{2-7} \cmidrule{3-7} \cmidrule{4-7} \cmidrule{5-7} \cmidrule{6-7} \cmidrule{7-7} 
 & $\varepsilon_{0}$ & $\text{CR}$ & $\varepsilon_{0}$ & $\text{CR}$ & $\varepsilon_{0}$ & $\text{CR}$\tabularnewline
\midrule 
$2^{-1}$ & 1.0748E-01 & -- & 3.4636E-02 & -- & 1.0946E-02 & -- \tabularnewline
\midrule 
$2^{-2}$ & 5.3665E-02 & 1.0020 & 1.0082E-02 & 1.7805 & 4.4932E-03 & 1.2846 \tabularnewline
\midrule 
$2^{-3}$ & 1.2131E-02 & 2.1453 & 3.6757E-03 & 1.4557 & 1.7822E-03 & 1.3341 \tabularnewline
\midrule 
$2^{-4}$ & 5.6863E-03 & 1.0931 & 1.3696E-03 & 1.4243 & 3.3765E-04 & 2.4001 \tabularnewline
\midrule 
$2^{-5}$ & 3.1386E-03 & 0.8574 & 3.1399E-04 & 2.1250 & 6.5368E-05 & 2.3689 \tabularnewline
\bottomrule
\end{tabular}
\caption{Relative errors and convergence rates for Example \ref{example2} at $T=0.016$ with varying parameters $(H,\ell)$.}
\label{tab:example2} 
\end{table}
In Table~\ref{tab:timeexample2}, we present the relative errors and their corresponding convergence rates for $\ell=1$, where the temporal convergence rates are much faster than expected (higher than 1) from theoretical results given in Theorem~\ref{thm:main-result}. As mentioned in Example~\ref{example1}, for a fixed coarse size $H$, the convergence rate in time also degrades as the spatial error gradually dominates the total error.
\begin{table}
\centering
\begin{tabular}{cccccc}
\toprule 
\multicolumn{6}{c}{Multiscale approximation at $T=0.016$}\tabularnewline
\midrule 
\multirow{1}{*}{$H$} & $\dt$ & $\varepsilon_{0}$ & $CR$ & $\varepsilon_{1}$ & $CR$\tabularnewline
\midrule 
$2^{-2}$ & $2^{-5}$ & 1.0082E--02 & -- & 2.0231E--02 & --\tabularnewline
\midrule 
$2^{-3}$ & $2^{-6}$ & 3.6751E--03  & 1.4559E+00 & 1.3950E--02 & 9.7428E--01\tabularnewline
\midrule 
$2^{-4}$ & $2^{-7}$ & 1.3686E--03 & 1.4251E+00  & 5.8310E--03 & 8.5800E--01\tabularnewline
\midrule 
$2^{-5}$ & $2^{-8}$ & 3.1213E--04  & 2.1325E+00 & 7.5210E--04 & 1.2340E+00 \tabularnewline
\bottomrule
\end{tabular} 
\caption{\label{tab:timeexample2} Relative errors and convergence rates in time, for Example~\ref{example2} at $T=0.016$ decreasing  simultaneously the time step size $\dt$ and the coarse size $H$ with $\ell=1$.} 
\end{table}

\begin{figure}[!t]
\centering
\subfloat[Multiscale solution.]{{\includegraphics[width=0.24\textwidth]{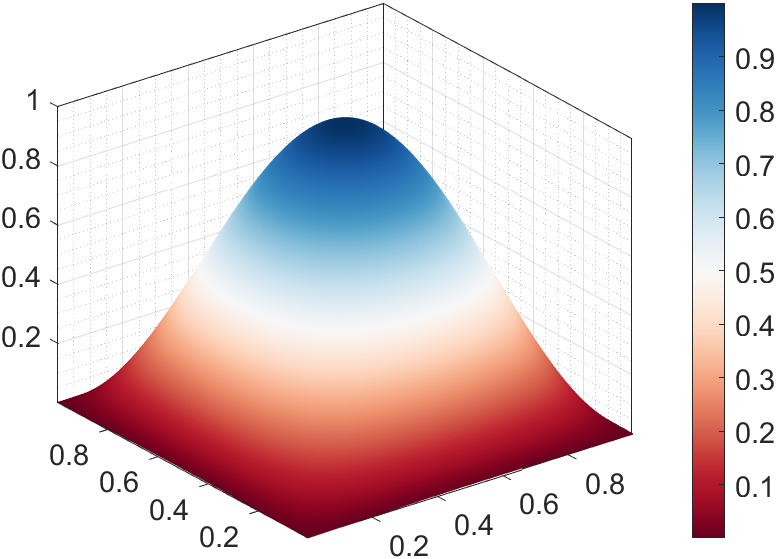}}~{\includegraphics[width=0.24\textwidth]{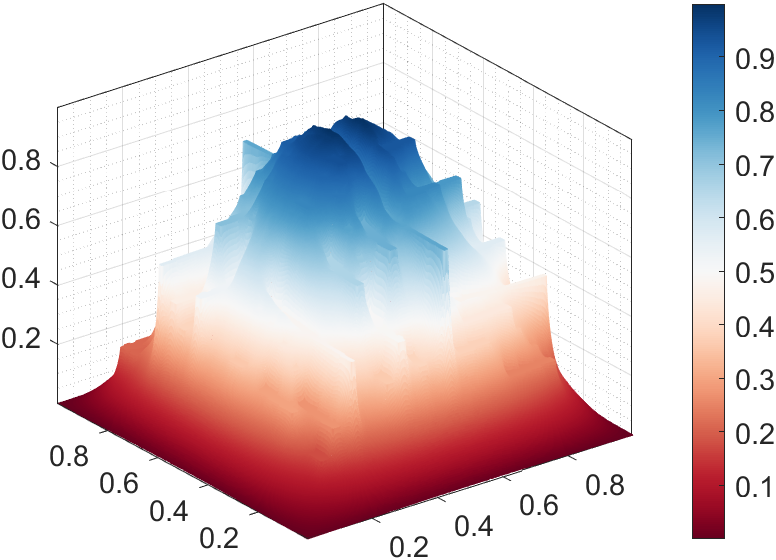}}~{\includegraphics[width=0.24\textwidth]{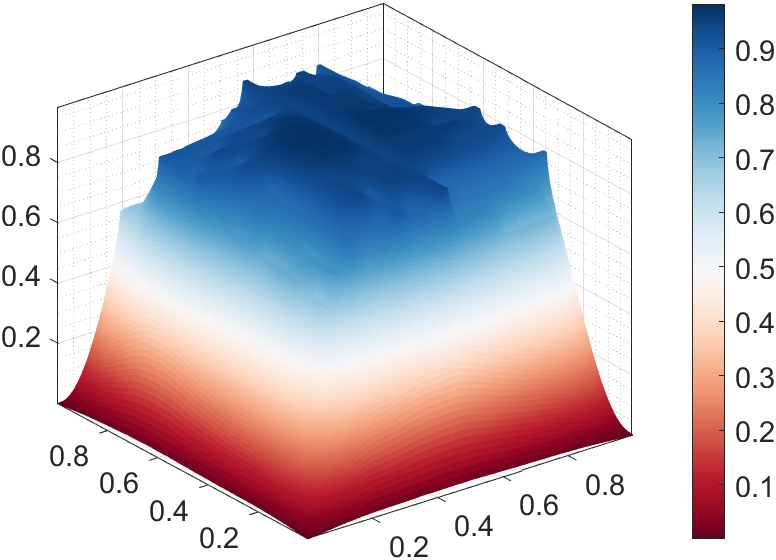}}~{\includegraphics[width=0.24\textwidth]{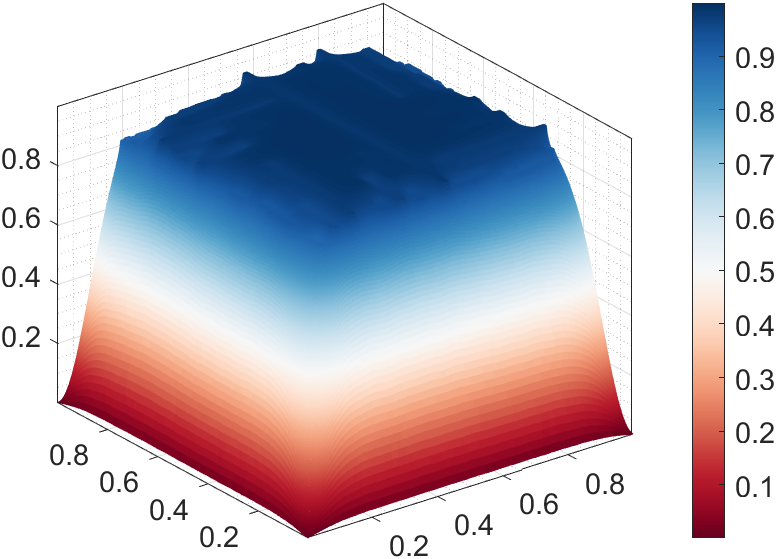}}}\\
\subfloat[exa1rf][Reference solution.]{{\includegraphics[width=0.24\textwidth]{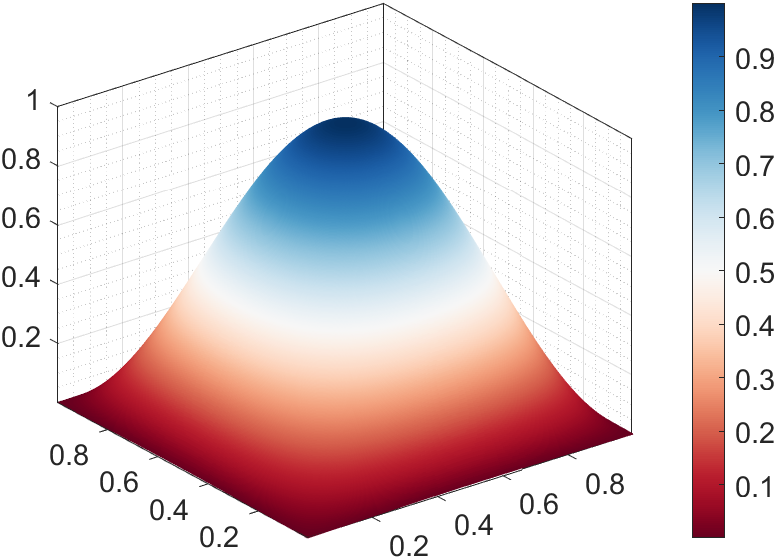}\label{fig:u0-ref-example2}}
~{\includegraphics[width=0.24\textwidth]{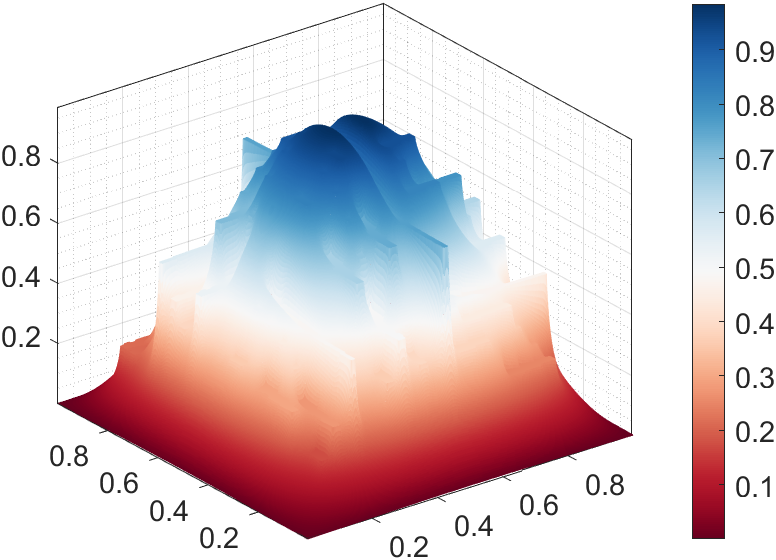}}~{\includegraphics[width=0.24\textwidth]{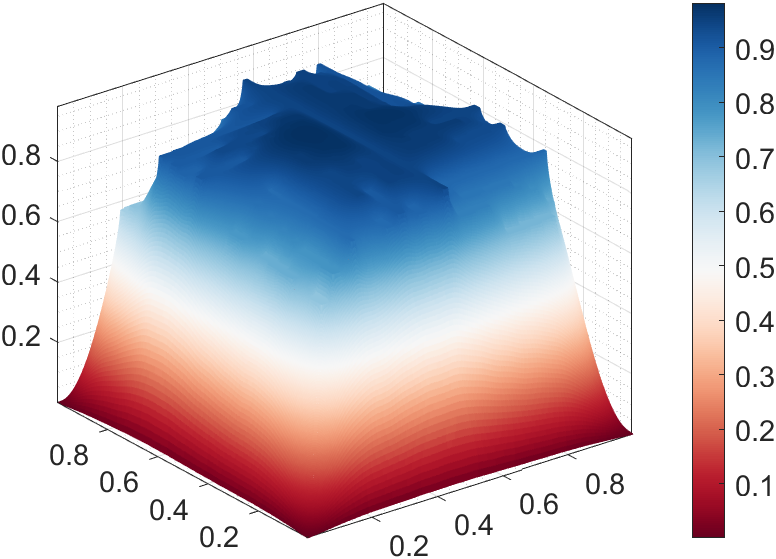}}~{\includegraphics[width=0.24\textwidth]{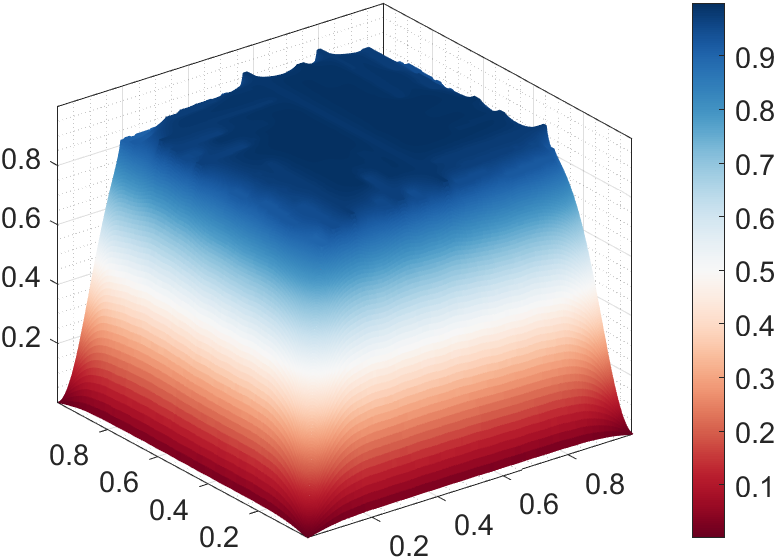}}}
\caption{Solutions of Example~\ref{example2} at $t=0$, $1.5625$E-04, $0.004$ and $0.012$ (a)Multiscale solution with $H=2^{-4}$ and $\ell=2$, and (b) the reference solution.}
\label{fig:example2}
\end{figure}

We present in Figure~\ref{fig:E-example2} the evolution of the classic free energy ${\cal E}$ and $\max$-norm of the multiscale solution of Example~\ref{example2}. Again, it is observed that the discrete energy decays monotonically over time, and the discrete maximum principle is preserved.

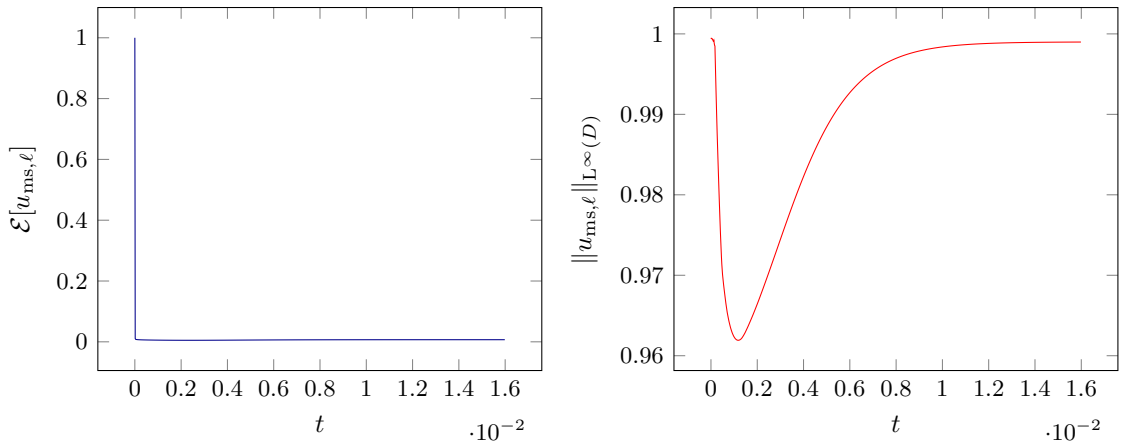
\begin{figure}[!h]
\centering
\begin{tikzpicture}
\begin{axis}[
grid=minor,
scale = 1.2,
xlabel={\small{$t$}}, 
ylabel={\small{${\cal E}[\umslew]$}},
]
\addplot[DarkBlue] table [x=T,y=EN] {EMei2d2_Hierarchical.dat};
\end{axis}
\end{tikzpicture}~
\begin{tikzpicture}
\begin{axis}[
grid=minor,
scale = 1.2,
xlabel={\small{$t$}}, 
ylabel={\small{$\|\umslew\|_{\LinfD}$}},
]
\addplot[Red] table [x=T,y=MAX] {EMei2d2_Hierarchical.dat};
\end{axis}
\end{tikzpicture}
\caption{\label{fig:E-example2} Time-dependent normalized discrete total energy (left) and maximum values of the multiscale solution $\umslew$ (right) over time evolution with $\dt=2^{-10}$ in Example~\ref{example2}.}
\end{figure}

\end{example}

\begin{example}
\label{example3}
In this numerical example, we consider a 3-d convective Allen-Cahn problem,
\begin{equation}
\label{eq:example3}
\pt u -\dive(\k_{3}\gd u)+\b\cdot\gd u = \frac{1}{\epsilon^{2}}(u-u^3),\quad\mbox{in }D\times(0,T],
\end{equation}
where $D=[0,1]^3$, $T=0.016$, and $\epsilon=0.1$. We also consider a highly heterogeneous permeability coefficient $\k_{3}$ shown in Figure~\ref{fig:perm3d-ex3} and the velocity field is $\b = (y-0.5,0.5 - x,0)^{T}$, 
the initial data is $u_{0}=0.1\sin(\pi x)\sin(\pi y)\sin(\pi z)$.

\begin{table}[!h]
\centering
\begin{tabular}{ccccc}
\toprule 
\multicolumn{5}{c}{Multiscale approximation at $T=0.016$}\tabularnewline
\midrule 
$H$ & $\varepsilon_{0}$ &  $\text{CR}$ & $\varepsilon_{1}$ & $\text{CR}$ \tabularnewline
\midrule 
$2^{-1}$ &  2.2542E-01 & -- &3.6893E-01 & -- \tabularnewline 	
\midrule 
$2^{-2}$ & 2.1299E-01 & 0.0606 & 3.5376E-01  & 0.0818 \tabularnewline
\midrule 
$2^{-3}$ & 5.1638E-02 & 2.0443 & 1.7974E-01 & 0.9768 \tabularnewline
\midrule 
$2^{-4}$ & 1.9051E-02 & 1.4386 & 7.2222E-02 & 1.3154 \tabularnewline
\bottomrule
\end{tabular}
\caption{ 
Relative errors and convergence rates, for Example \ref{example3} at $T=0.016$ with $\ell=1$ and varying mesh size $H$.
}
\label{tab:example3}

\end{table}

Numerical results for \eqref{eq:example3} with the fixed level parameter of $\ell=1$ on each coarse neighborhood with varying coarse grid size are reported in Table~\ref{tab:example3}. Convergence with respect to coarse grid size $H$ and wavelet level $\ell$ is observed as expected. We notice that the spatial accuracy in the $\Lt$-norm is just about $1$ as expected. Figure~\ref{fig:example3} depicts the profile of the reference and multiscale solution for problem~\eqref{eq:example3} at different instants $t=0$, $0.004$, $0.008$ and $0.016$, which visually verifies the accuracy of our multiscale methods. 

\begin{figure}[!t]
\centering
\subfloat[Multiscale solution.]{{\includegraphics[width=0.24\textwidth]{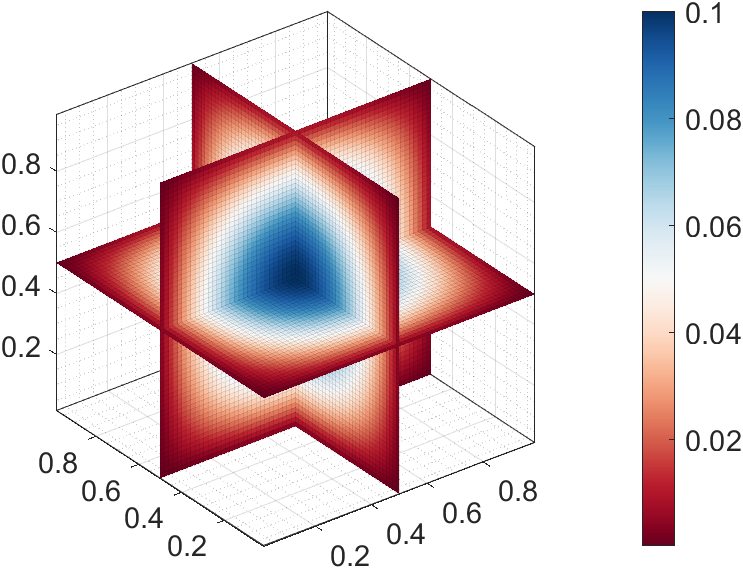}}~{\includegraphics[width=0.24\textwidth]{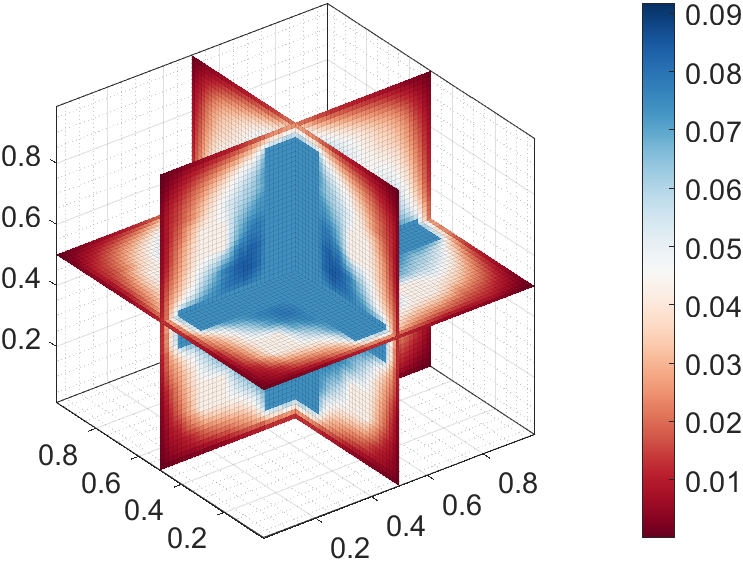}}~{\includegraphics[width=0.24\textwidth]{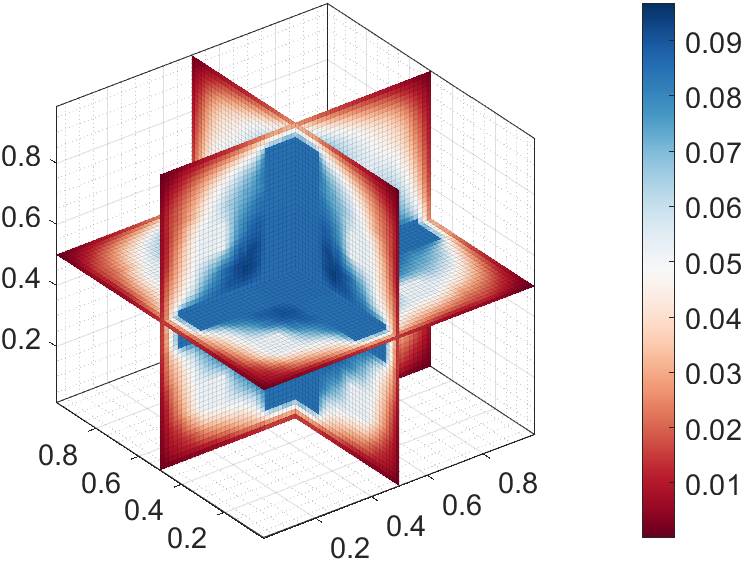}}~{\includegraphics[width=0.24\textwidth]{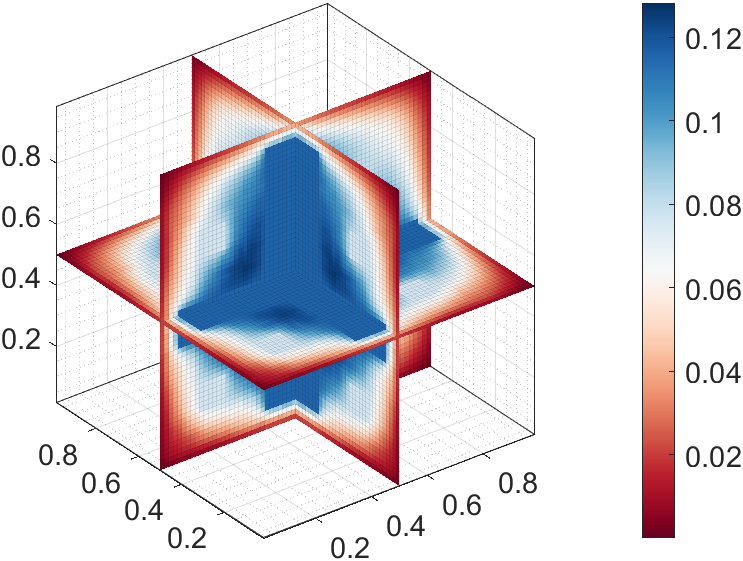}}}\\
\subfloat[exa1rf][Reference solution.]{{\includegraphics[width=0.24\textwidth]{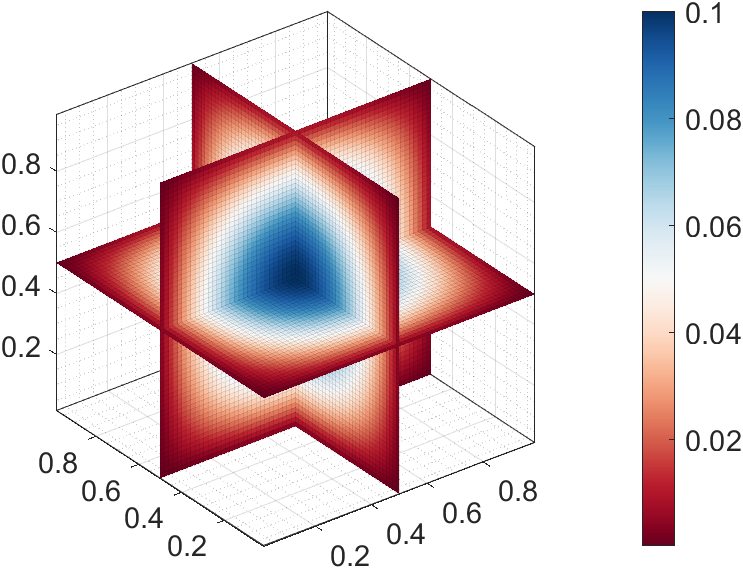}}
~{\includegraphics[width=0.24\textwidth]{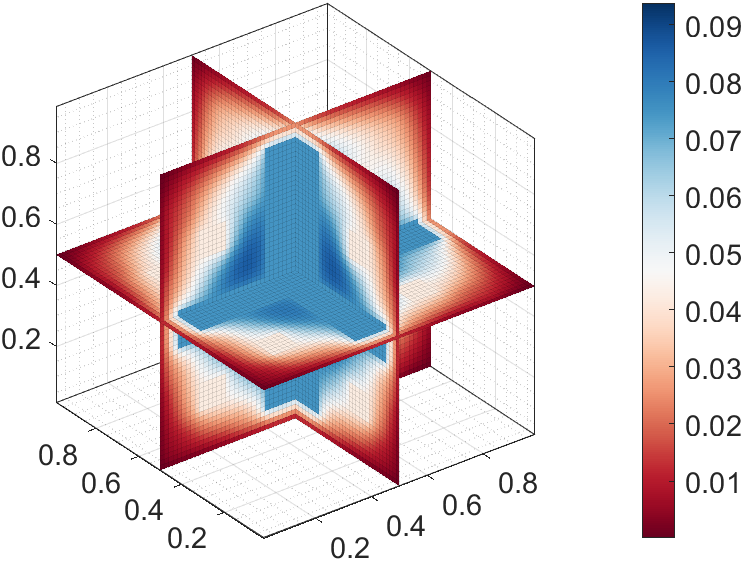}}~{\includegraphics[width=0.24\textwidth]{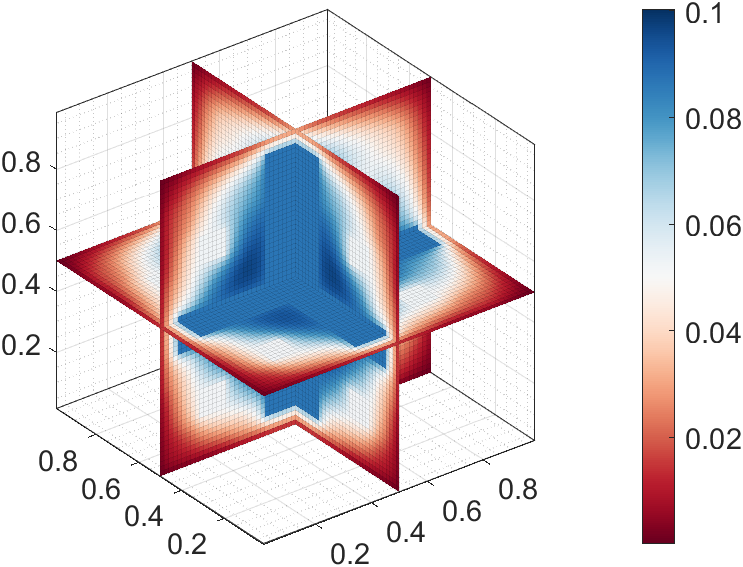}}~{\includegraphics[width=0.24\textwidth]{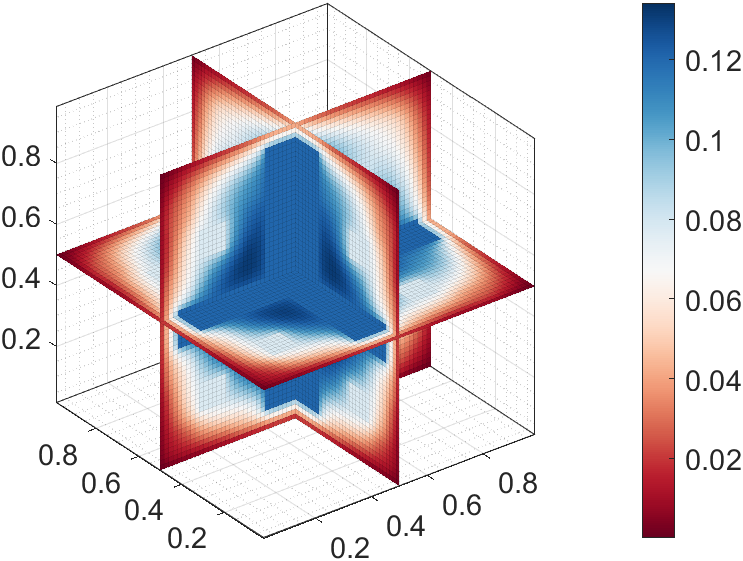}}}
\caption{\label{fig:example3} Solutions of Example~\ref{example3} at instants $t=0$, $0.004$, $0.008$ and $0.016$ (a) Multiscale solution using coarse size of $H=2^{-4}$ and level parameter $\ell=2$, and (b) the reference solution.} 
\end{figure}

\end{example}

\begin{example}
\label{example1s}
In this test, we consider the Schnakenberg reactive model with constant convection, which studies the limit cycle behaviors of two-component chemical reactions \citep{schnakenberg1979simple,montanelli2020solving}:  
\begin{equation}
\label{eq:example1s}
\begin{aligned}
\pt \uo -\dive(\k_{1}\gd\uo)+\bo\cdot\gd\uo& = \gamma(a-\uo+\uo^{2}\ut),&&\quad\mbox{in }D\times(0,T]\\
\pt \ut -\dive(\k_{2}\gd\ut)+\bt\cdot\gd\ut& = \gamma(b-\uo^{2}\ut),&&\quad\mbox{in }D\times(0,T]
\end{aligned}
\end{equation}
where $\uo$ and $\ut$ represent the concentrations of the chemical species assumed to be maintained at constant concentrations $a$ and $b$. We take $D:=[0,30]\times[0,30]$, $T=5$, $\gamma=3$, $a=0.1$ and $b=0.9$. $\k_{1}=1$, $\k_{2}=10$ and $\bo = \bt = [-1,-1]^{T}$ are constant coefficients and constant velocities. The initial data is
\begin{align}
\uo(\x,0)&=1-e^{-2[(x-L_{x}/2.15)^{2}+(y-L_{x}/2.15)^{2}]},\\
\ut(\x,0)&=\frac{b}{(a+b)^{2}}-e^{-2[(x-L_{x}/2)^{2}+(y-L_{x}/2)^{2}]},
\end{align}
with $L_{x}=30$. The simulation is run for $2^{9}$ time steps. 

In Table \ref{tab:example1s}, we present the relative errors corresponding to varying coarse mesh size for the problem~\eqref{eq:example1s} at the final time $T$. We observe that both relative errors decay as $\ell$ increases.


Figure~\ref{fig:example1s} depicts the profile of the references and multiscale solutions using our approach to approximate the two-component coupled system \eqref{eq:example1s}. 
We can see that the multiscale method can capture most small-scale details of the reference solution.
\begin{table}
\centering
\centering
\begin{tabular}{ccccccc}
\toprule 
\multicolumn{7}{c}{Multiscale approximation for $\uo$ at $T=5$}\tabularnewline
\midrule 
\multirow{2}{*}{$H$} & \multicolumn{2}{c}{$\ell=0$} & \multicolumn{2}{c}{$\ell=1$} & \multicolumn{2}{c}{$\ell=2$}\tabularnewline
\cmidrule{2-7} \cmidrule{3-7} \cmidrule{4-7} \cmidrule{5-7} \cmidrule{6-7} \cmidrule{7-7} 
 & $\varepsilon_{1}$ & $\varepsilon_{0}$ & $\varepsilon_{1}$ & $\varepsilon_{0}$ & $\varepsilon_{1}$ & $\varepsilon_{0}$\tabularnewline
\midrule 
$2^{-1}$ & 1.6903E+00 & 6.3640E-01 & 1.2440E+00 & 4.6745E-01 & 8.8215E-01 & 3.0306E-01 \tabularnewline
\midrule 
$2^{-2}$ & 6.324E3-01 & 1.9572E-01 & 5.4738E-01 & 1.4371E-01 & 3.2769E-01 & 7.0549E-02 \tabularnewline
\midrule 
$2^{-3}$ & 4.0562E-01 & 9.5299E-02 & 2.1348E-01 & 4.0011E-02 & 1.2758E-01 & 2.1181E-02 \tabularnewline
\midrule 
$2^{-4}$ & 1.4803E-01 & 2.0796E-02 & 4.9016E-02 & 5.4210E-03 & 2.0146E-02 & 3.4661E-03 \tabularnewline
\midrule 
$2^{-5}$ & 3.5231E-02 & 3.8360E-03 & 1.2758E-02 & 3.3414E-03 & 9.1274E-03 & 3.3328E-03 \tabularnewline

\toprule 
\multicolumn{7}{c}{Multiscale approximation for $\ut$ at $T=5$}\tabularnewline
\midrule 
\multirow{2}{*}{$H$} & \multicolumn{2}{c}{$\ell=0$} & \multicolumn{2}{c}{$\ell=1$} & \multicolumn{2}{c}{$\ell=2$}\tabularnewline
\cmidrule{2-7} \cmidrule{3-7} \cmidrule{4-7} \cmidrule{5-7} \cmidrule{6-7} \cmidrule{7-7} 
 & $\varepsilon_{1}$ & $\varepsilon_{0}$ & $\varepsilon_{1}$ & $\varepsilon_{0}$ & $\varepsilon_{1}$ & $\varepsilon_{0}$\tabularnewline
\midrule 
$2^{-1}$ & 7.4820E-01 & 4.2152E-01 & 5.3147E-01 & 2.9112E-01 & 3.8746E-01 & 1.6973E-01 \tabularnewline
\midrule 
$2^{-2}$ & 2.6432E-01 & 1.3413E-01 & 1.6867E-01 & 7.5735E-02 & 7.6502E-02 & 2.5090E-02 \tabularnewline
\midrule 
$2^{-3}$ & 1.2343E-01 & 4.7182E-02 & 5.1759E-02 & 1.9423E-02 & 2.4450E-02 & 9.3365E-03 \tabularnewline
\midrule 
$2^{-4}$ & 2.9775E-02 & 9.4506E-03 & 8.8690E-03 & 1.9526E-03 & 4.4339E-03 & 1.5533E-03 \tabularnewline
\midrule 
$2^{-5}$ & 7.8642E-03 & 1.6049E-03 & 4.1623E-03 & 1.5737E-03 & 3.7082E-03 & 1.4877E-03 \tabularnewline
\bottomrule
\end{tabular}
\caption{\label{tab:example1s}Relative errors for Example~\ref{example1s} at $T=5$ with varying parameters $(H,\ell)$.}
\end{table}

\begin{figure}[!h]
\centering
\subfloat[$\ell=0$.]{{\includegraphics[width=0.25\textwidth]{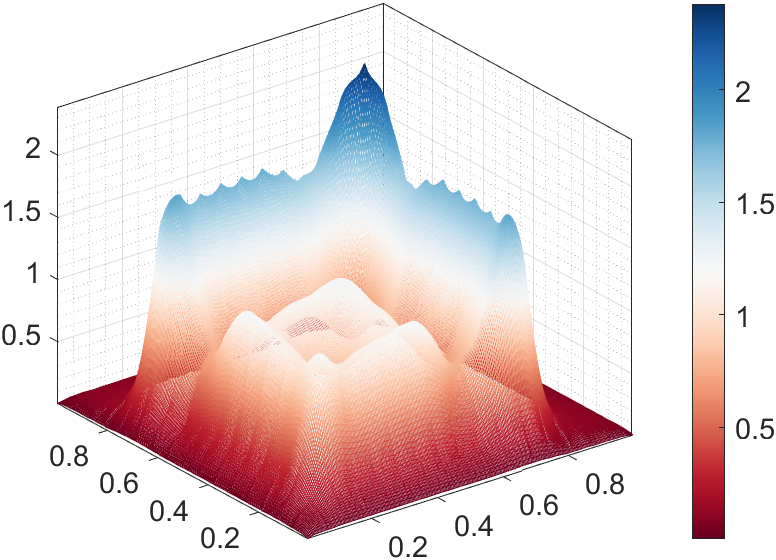}}~{\includegraphics[width=0.25\textwidth]{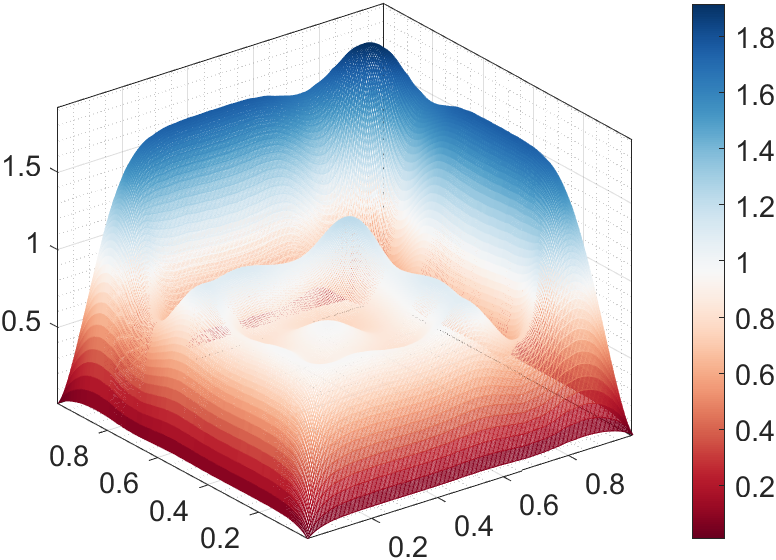}}}~
\subfloat[$\ell=1$.]{{\includegraphics[width=0.25\textwidth]{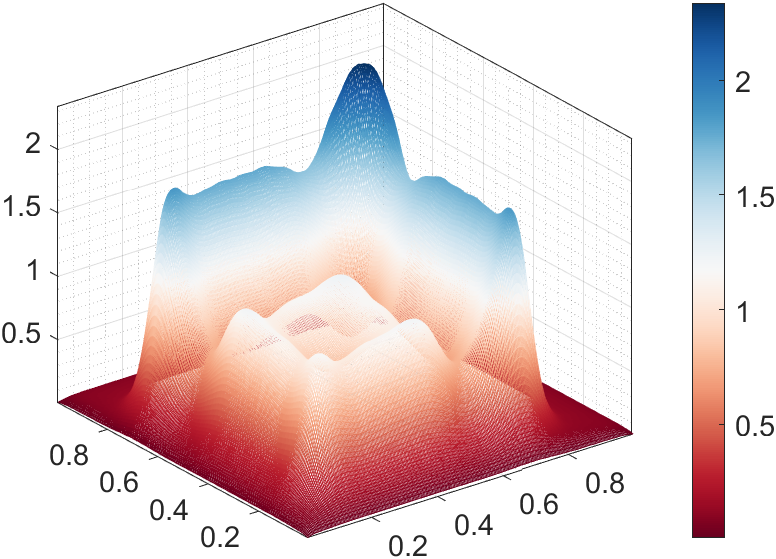}}~{\includegraphics[width=0.25\textwidth]{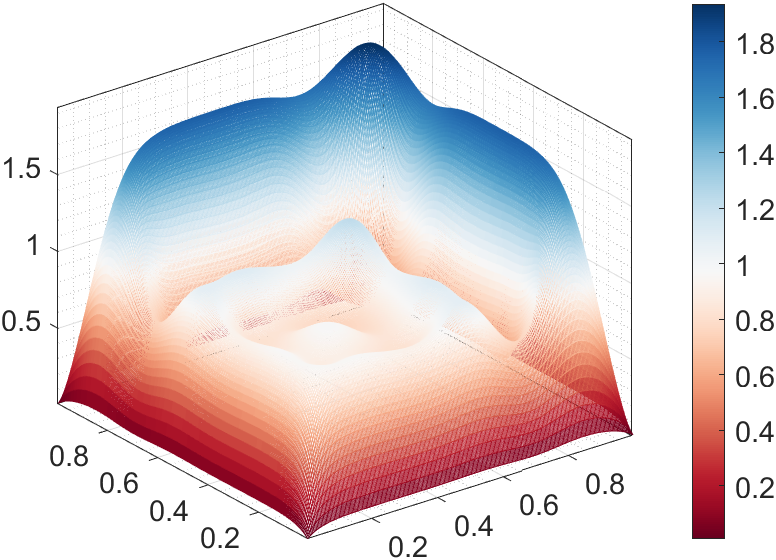}}}\\
\subfloat[$\ell=2$.]{{\includegraphics[width=0.25\textwidth]{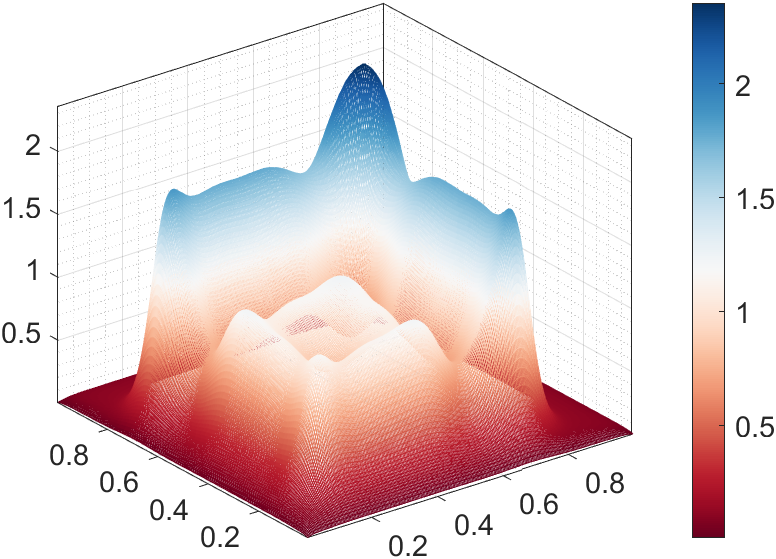}}~{\includegraphics[width=0.25\textwidth]{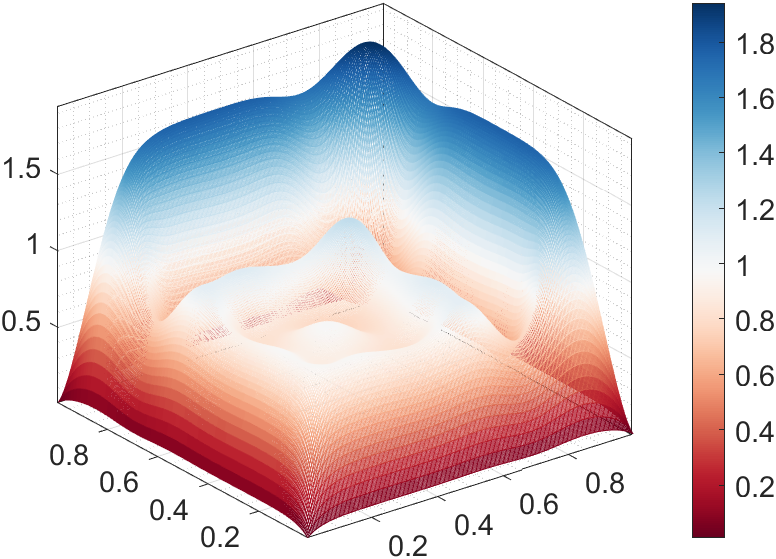}}}~
\subfloat[exa1rf][Reference solution.]{{\includegraphics[width=0.25\textwidth]{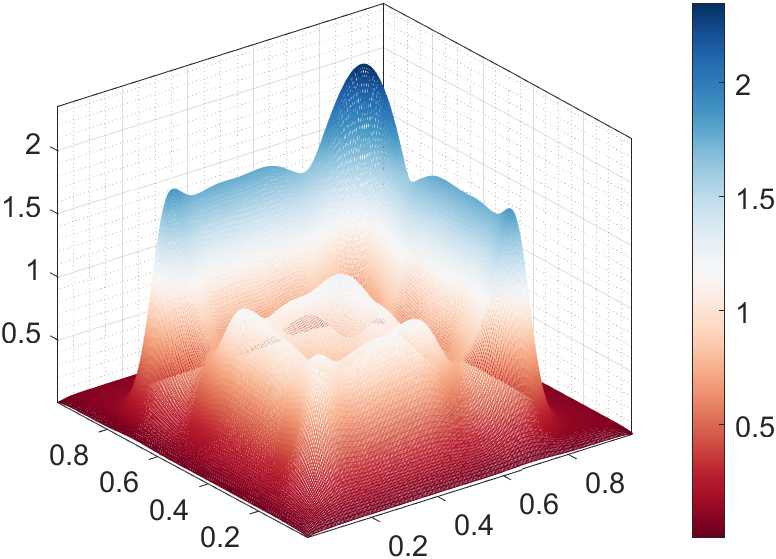}}
~{\includegraphics[width=0.25\textwidth]{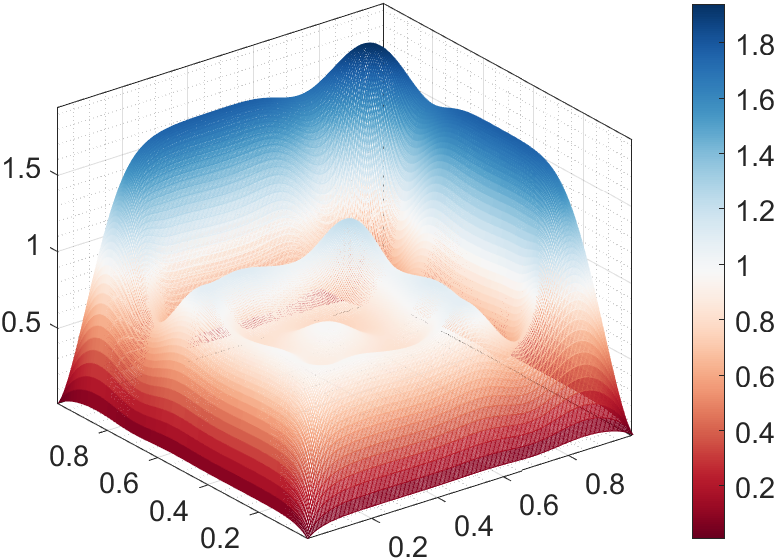}}}
\caption{\label{fig:example1s} Multiscale solution with $2^{9}$ time-steps with $H=2^{-4}$ and (a) $\ell=0$, (b) $\ell=1$ and (c) $\ell=2$,  and (d) the reference solution of Example~\ref{example1s} at final time $T=5$.}
\end{figure}
\end{example}

\begin{example}
\label{example2s}
The last example is to investigate our method for the following Schnakenberg reactive problem with heterogeneous coefficient:
\begin{equation}
\label{eq:example2s}
\begin{cases}
\pt \uo -\dive(\k_{1}\gd\uo)+\bo\cdot\gd\uo& = \ut(1-\uo),\\
\pt \ut -\dive(\k_{2}\gd\ut)+\bt\cdot\gd\ut& = \uo(1-\ut^{2}),
\end{cases}
\end{equation}
the permeability coefficients $\k_{1}$ and $\k_{2}$ are heterogeneous and high-contrast with values of $10^{2}$ in the high-contrast channel constant. Figures~\ref{fig:perm-example2s}--\ref{fig:perm-example2} plot the permeability fields $\k_{1}$ and, $\k_{2}$ respectively. Further, for this example.  The velocities $\bo=\bt$ are given as in the Example~\ref{example2}, while the initial conditions are defined as $\uo(\x,0) = \sin(3\pi x)\sin(2\pi y)$ and  $\ut(\x,0) = \sin(2\pi x)\cos(3\pi y)$. 
The terminal time of this simulation is $T=0.01$, and there are $2^{9}$ time steps. Table \ref{tab:example2s} lists the relative errors with respective to varying parameters $(H,\ell)$ at the terminal time $T$. Similar to Example \ref{example1s}, we notice that the errors in levels $1$ and $2$ decay faster than the first $\ell=0$ as expected for two solutions in both norms.

Finally, Figure~\ref{fig:example2s} depicts the profile of the reference solutions and multiscale solutions at the terminal time $T=0.01$. We observe that the multiscale solutions can capture the multiple scales in the reference solutions. 
\begin{table}
\centering
\subfloat
\centering
\begin{tabular}{ccccccc}
\toprule 
\multicolumn{7}{c}{Multiscale approximation for $\uo$ at $T=0.01$}\tabularnewline
\midrule 
\multirow{2}{*}{$H$} & \multicolumn{2}{c}{$\ell=0$} & \multicolumn{2}{c}{$\ell=1$} & \multicolumn{2}{c}{$\ell=2$}\tabularnewline
\cmidrule{2-7} \cmidrule{3-7} \cmidrule{4-7} \cmidrule{5-7} \cmidrule{6-7} \cmidrule{7-7} 
 & $\varepsilon_{1}$ & $\varepsilon_{0}$ & $\varepsilon_{1}$ & $\varepsilon_{0}$ & $\varepsilon_{1}$ & $\varepsilon_{0}$\tabularnewline
\midrule 
$2^{-1}$ & 6.3216E-01 & 5.5896E-01 & 2.5387E-01 & 1.4123E-01 & 1.0218E-01 & 3.6892E-02 \tabularnewline
\midrule 
$2^{-2}$ & 3.0683E-01 & 1.8847E-01 & 1.7760E-01 & 6.9064E-02 & 1.3019E-01 & 4.0238E-02 \tabularnewline
\midrule 
$2^{-3}$ & 2.4723E-01 & 1.2302E-01 & 1.2478E-01 & 3.4178E-02 & 8.9496E-02 & 1.7929E-02 \tabularnewline
\midrule 
$2^{-4}$ & 1.8703E-01 & 7.3334E-02 & 7.4127E-02 & 1.1668E-02 & 4.4538E-02 & 3.6601E-03 \tabularnewline
\midrule 
$2^{-5}$ & 1.3697E-01 & 3.8714E-02 & 4.4229E-02 & 2.9354E-03 & 1.9524E-02 & 5.1187E-03 \tabularnewline

%
\toprule 
\multicolumn{7}{c}{Multiscale approximation for $\ut$ at $T=0.01$}\tabularnewline
\midrule 
\multirow{2}{*}{$H$} & \multicolumn{2}{c}{$\ell=0$} & \multicolumn{2}{c}{$\ell=1$} & \multicolumn{2}{c}{$\ell=2$}\tabularnewline
\cmidrule{2-7} \cmidrule{3-7} \cmidrule{4-7} \cmidrule{5-7} \cmidrule{6-7} \cmidrule{7-7} 
 & $\varepsilon_{1}$ & $\varepsilon_{0}$ & $\varepsilon_{1}$ & $\varepsilon_{0}$ & $\varepsilon_{1}$ & $\varepsilon_{0}$\tabularnewline
\midrule 
$2^{-1}$ & 4.3368E-01 & 2.9545E-01 & 1.8295E-01 & 9.2601E-02 & 9.1993E-02 & 2.9765E-02 \tabularnewline
\midrule 
$2^{-2}$ & 2.1568E-01 & 8.3863E-02 & 1.5276E-01 & 4.3575E-02 & 1.3065E-01 & 4.2853E-02 \tabularnewline
\midrule 
$2^{-3}$ & 1.8089E-01 & 5.6617E-02 & 1.0221E-01 & 1.9836E-02 & 7.3485E-02 & 1.0946E-02 \tabularnewline
\midrule 
$2^{-4}$ & 1.2553E-01 & 2.7039E-02 & 5.8860E-02 & 6.0603E-03 & 3.6188E-02 & 2.3770E-03 \tabularnewline
\midrule 
$2^{-5}$ & 1.1000E-01 & 1.8510E-02 & 3.7821E-02 & 1.9223E-03 & 1.5972E-02 & 2.8104E-03 \tabularnewline
\bottomrule
\end{tabular}
\caption{\label{tab:example2s}  Relative errors for problem~\eqref{eq:example2s} at the terminal time $T=0.01$ with varying parameters $(H,\ell)$.}
\end{table}

\begin{figure}[!h]
\centering
\subfloat[$\ell=0$.]{{\includegraphics[width=0.25\textwidth]{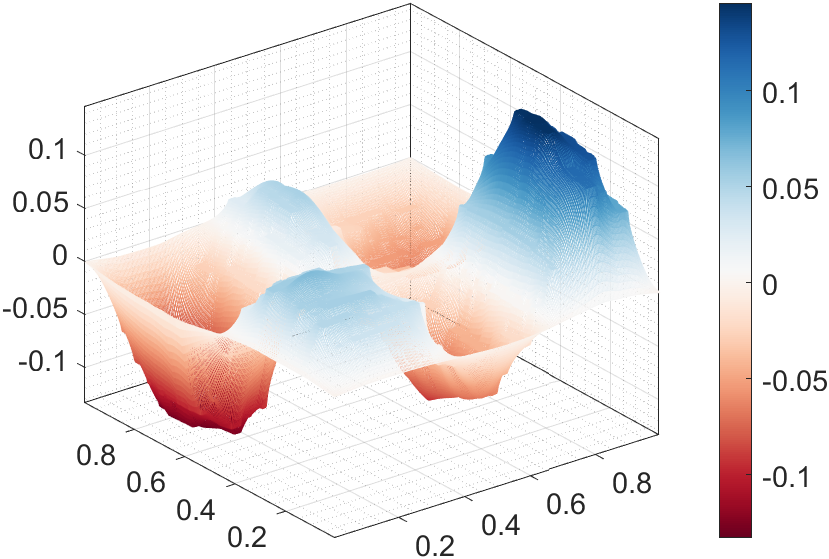}}~{\includegraphics[width=0.25\textwidth]{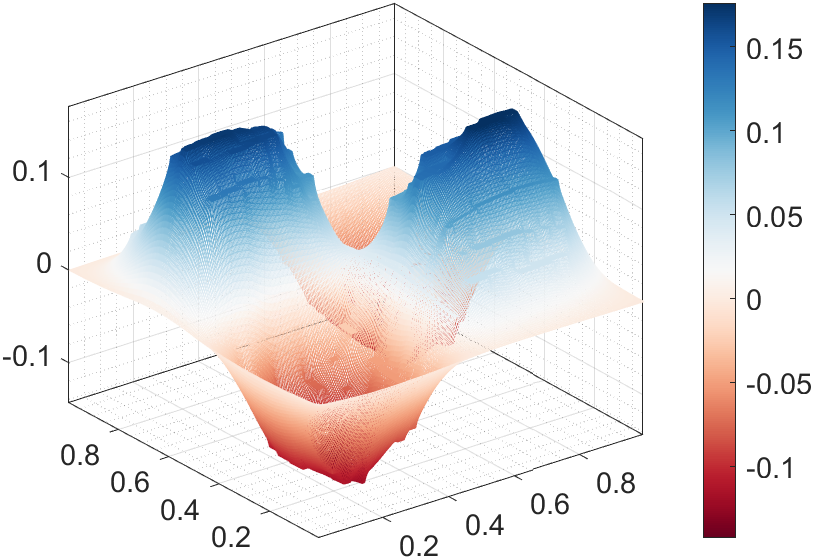}}}~\subfloat[$\ell=1$.]{{\includegraphics[width=0.25\textwidth]{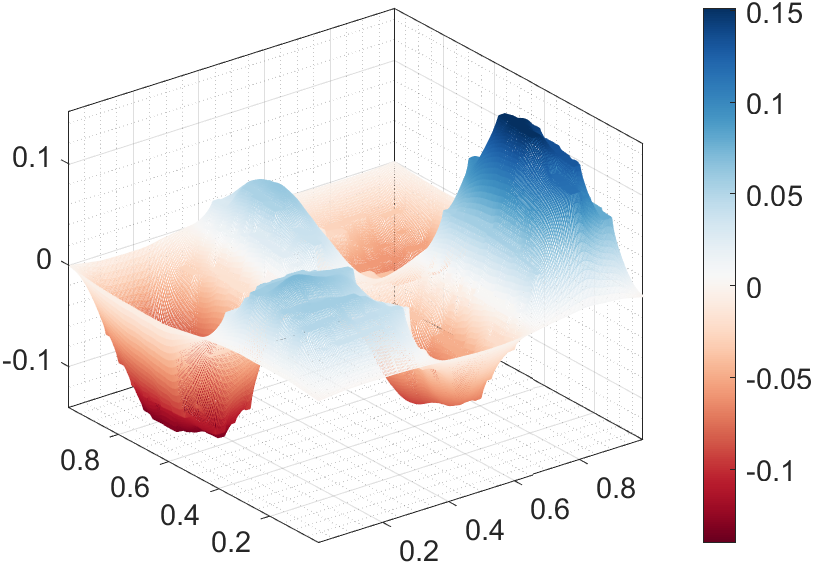}}~{\includegraphics[width=0.25\textwidth]{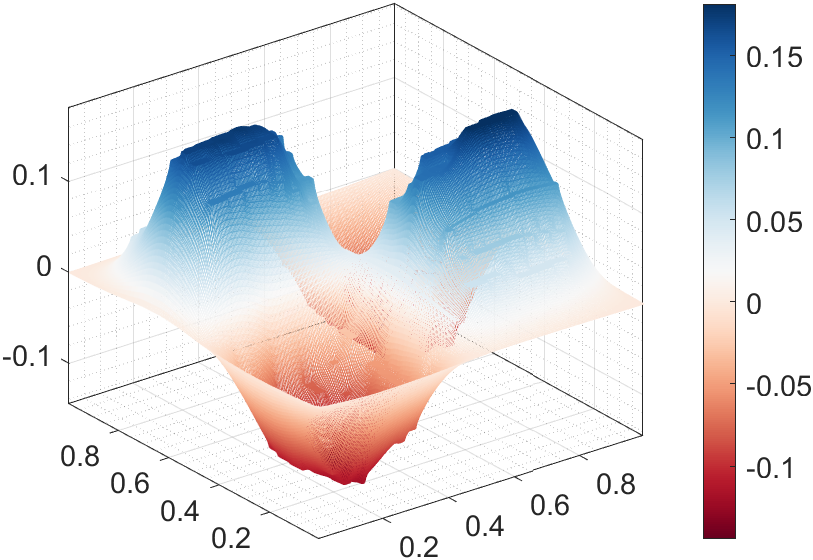}}}\\
	\subfloat[$\ell=2$.]{{\includegraphics[width=0.25\textwidth]{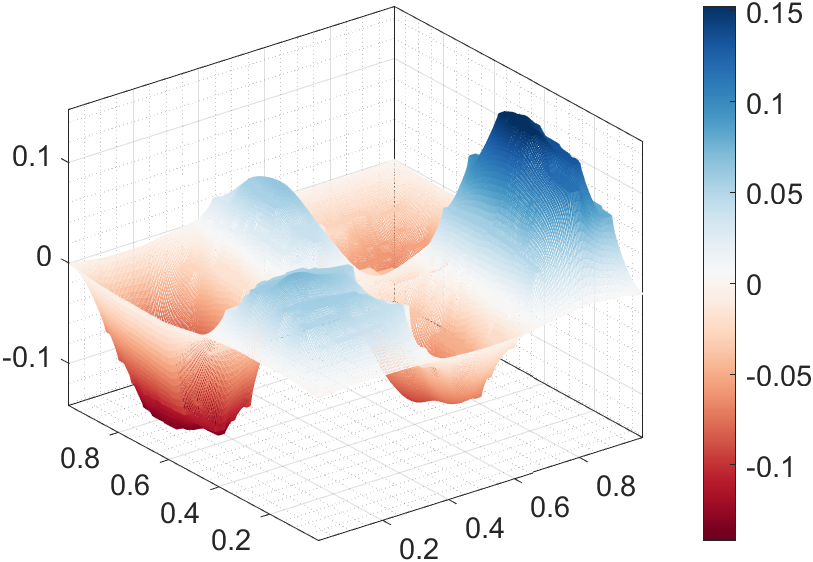}}~{\includegraphics[width=0.25\textwidth]{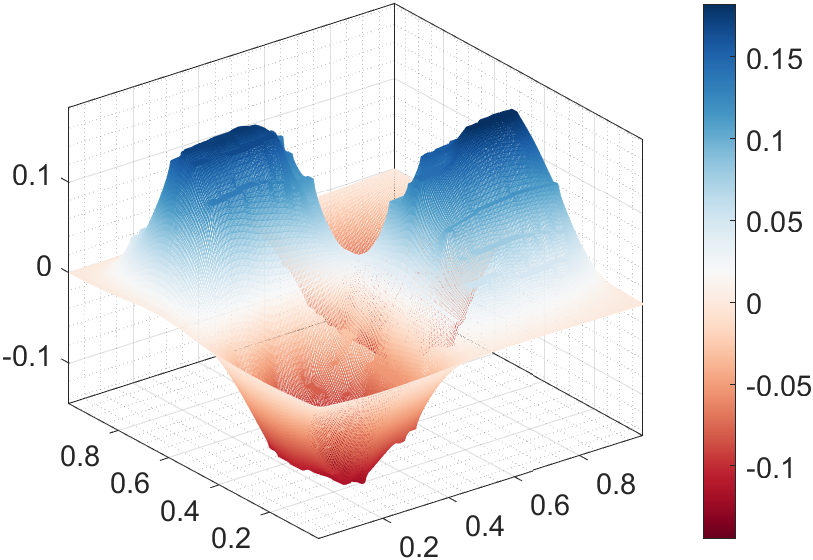}}}~
\subfloat[Reference solution.]{{\includegraphics[width=0.25\textwidth]{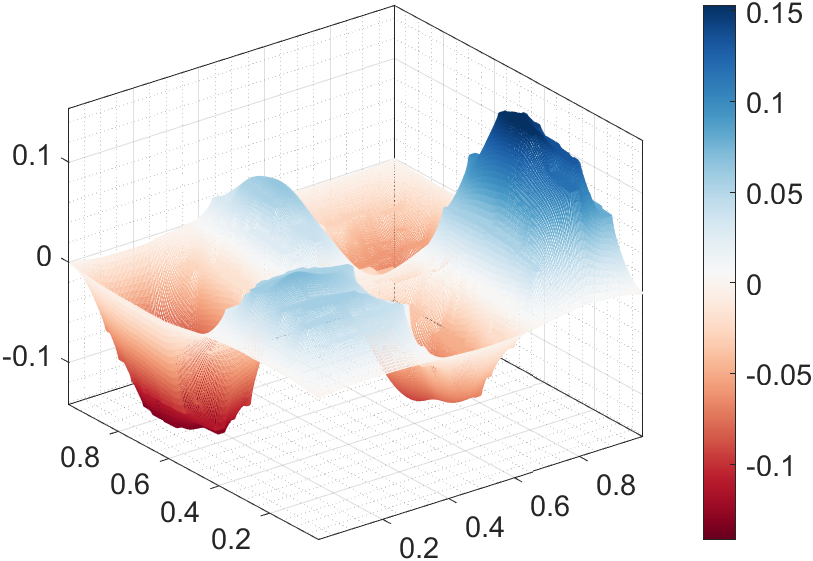}}~{\includegraphics[width=0.25\textwidth]{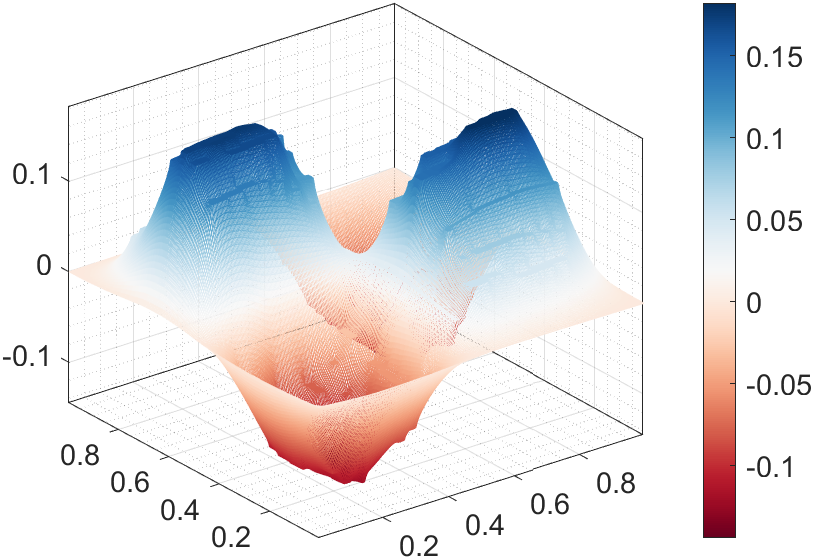}}}
\caption{\label{fig:example2s} Multiscale solution with $2^{9}$ time-steps using $H=2^{-4}$ and (a) $\ell=0$, (b) $\ell=1$ and (c) $\ell=2$,  and (d) the reference solution of problem~~\eqref{eq:example2s} at final time $T=0.01$.}
\end{figure}

\end{example}
%
\section{Conclusions}
\label{sec:conclusion}
We have developed an efficient multiscale method to treat the semilinear parabolic problems with a heterogeneous coefficient. The approximation properties of the corresponding steady-state multiscale ansatz space are derived. Since the nonlinear term lies in the reaction term, we can develop a time-independent multiscale ansatz space in this work. The case with the quasilinear term has many important applications, \eg, in the metamaterial, and is more computationally involved and probably requires time-dependent, multiscale ansatz space. We plan to investigate it in the future. 

\section*{Acknowledgement}
The research of Guanglian Li is partially supported by the Hong
Kong RGC Early Career Scheme (Project: 27301921). The research of Eric Chung is partially supported by the Hong Kong RGC General Research Fund (Project: 14304021).

\bibliographystyle{plainnat}
\bibliography{ref_wavelet_edge}
\end{document}